\documentclass[a4paper,11pt]{article}
\usepackage[utf8]{inputenc}
\usepackage[T1]{fontenc}

\usepackage[affil-it]{authblk}
\usepackage{indentfirst}
\usepackage{enumitem}
\usepackage{amsmath}
\usepackage{amssymb,amsmath,amsthm}
\usepackage{ebproof}
\usepackage{comment}
\usepackage{aliascnt}
\usepackage{hyperref}
\usepackage[capitalise]{cleveref}
\usepackage{ragged2e}
\usepackage{adjustbox}
\usepackage[style=alphabetic,maxalphanames=4,minalphanames=4]{biblatex}
\usepackage [autostyle, english = american]{csquotes}
\MakeOuterQuote{"}

\usepackage{tikz-cd}
\usetikzlibrary{decorations.markings,intersections}
\usetikzlibrary{matrix}
\usetikzlibrary{arrows}
\usetikzlibrary{arrows.meta}
\usetikzlibrary{decorations.pathmorphing,shapes}
\usetikzlibrary{decorations.markings}
\tikzset{shorten <>/.style={shorten >=#1,shorten <=#1}}
\tikzset{every picture/.prefix code=\DisableQuotes}

\tikzset{spanmap/.style={
            decoration={markings,
            mark= at position 0.5 with {
                  \node[transform shape] (tempnode) {$|$};

                  }
              },
              postaction={decorate}
}
}

\usepackage{quiver}
\usepackage{xcolor}

\usepackage[mathcal]{euscript}

\usepackage{amssymb}
\usepackage{ stmaryrd }
\usepackage{ dsfont }
\usepackage{upgreek}
\usepackage{ mathrsfs }
\usepackage[bb=dsserif]{mathalpha}
\usepackage{relsize}

\let\oldtheorem\newtheorem
\RenewDocumentCommand{\newtheorem}{s m o m O{}}{%
\IfBooleanTF{#1}%
{\oldtheorem{#2}{#4}}%
{\IfNoValueTF{#3}{\oldtheorem{#2}{#4}[#5]}%
{\newaliascnt{#2}{#3}%
\oldtheorem{#2}[#2]{#4}%
\aliascntresetthe{#2}}}}

\title{Non-linear contexts and dependent operads}
\author{El Mehdi Cherradi}
\affil{IRIF - Inria}
\date{}

\newtheorem{theorem}{Theorem}[section]
\newtheorem{proposition}[theorem]{Proposition}
\newtheorem{lemma}[theorem]{Lemma}
\newtheorem{corollary}[theorem]{Corollary}

\theoremstyle{definition}
\newtheorem{definition}[theorem]{Definition}

\theoremstyle{remark}
\newtheorem{remark}[theorem]{Remark}
\newtheorem{example}[theorem]{Example}

\begin{document}

\maketitle

\begin{abstract}
 We study the category $\mathbf{Dep_{aff}}$ of finite posets and opfibrations between them, and prove that it is a perfect operator category in the sense of Barwick. Based on its additional properties, we propose a definition for a dependent operator category, and formulate a notion of dependent affine operad in a fibered way. In the particular case of $\mathbf{Dep_{aff}}$, we investigate the basic theory of such operads, establishing that the category of algebras for an operad is monadic over a canonical presheaf topos of "dependent sets". We finally extend this theory to a setting where we allow variable duplication, which we refer to as general dependent operads.
\end{abstract}

\tableofcontents

\newpage 

\section*{Introduction}
\addcontentsline{toc}{section}{Introduction}

The starting point of the present work is the fibered definition for operads, as introduced by Lurie \cite{lurie2017higher}, and the notion of operator category defined by Barwick \cite{barwick2018operator} that abstracts away the main features of the category $\mathbf{FinSet}_*$ of pointed finite sets used by Lurie in his definition.

The observation is then the following: while an operator category expresses operations that ought to be defined by their points (that is, global sections of an object $X$ of an operator category play a key role in the theory), we may want to consider notions of arity that do not reduce to points. In particular, in dependent type theory, types may depend on other types so that forgetting the dependency does not make sense in general (i.e., a dependent type cannot directly be thought of as a "single point", but rather as an "arrow").

Despite the key role played by dependencies in dependent type theory, one may have the impression that its formalism stops halfway through in that it only considers linear contexts, that is, a finite list of types where each type can depend on the previous ones. Of course, thanks to substitutions, we may still be able to express situations where a type depends on two independent types: for example, if $A$ and $B$ are types in the empty context, we can define a type $C$ in context $A, p_A^* B$, where $A, p_A^* B$ is the context obtained by substituting $B$ along the projection $A \to \emptyset$ from the context with only the type $A$ to the empty context. However, not only does this way of proceeding break the natural symmetry (forcing us to choose between $A, p_A^* B$ and $B, p_B^* A$ as the context for $C$), but it seems to forget the precise nature of dependency between $A$, $B$ and $C$. In particular, there is no distinction between "real" dependency and dependency created via substitution along projections, which feels like an artificially forced kind of dependency.

For these reasons, it seems very reasonable to refuse to collapse arbitrary dependency relations, which we believe can be expressed better via a finite poset rather than via a finite linear order. This suggests studying a category of finite posets as a candidate for a \textit{dependent} operator category where finite posets $P$ play the role of arities. In this setting, our claim is that the natural notion of morphism to consider consists of opfibrations, accounting for the fact that, in particular, sub-contexts of a context of shape $P$ should gather all the dependencies of the types they pick (i.e. such inclusions should form cosieves).

An important thing to note regarding the terminology is that what we call "affine dependent operads" in this document really generalizes the notion of colored PROPs (i.e. we allow not only a poset-shaped context for the input, but also the same for the output). Moreover, our proposed notion of "general" dependent operads adds variable duplication compared to the affine case, which is usually something that is not considered in the context of "non-dependent" operads, but seems like the more important notion in the dependent case (for instance, categories cannot be expressed in an affine way because the construction of identity morphisms "duplicates" the input object, making it available both as the domain and the codomain of the corresponding identity arrow).

\section*{Outline}

In the first section, we introduce and study the category $\mathbf{Dep_{aff}}$ of finite posets and opfibrations. Interestingly, it enjoys key properties that were not needed in the case of operator categories in the sense of Barwick, such as regularity.

In the second section, we prove that $\mathbf{Dep_{aff}}$ is a perfect operator category in the sense of Barwick, and propose a definition for a (perfect) dependent operator category with the idea that $\mathbf{Dep_{aff}}$ plays, in the dependent setting, the role that $\mathbf{FinSet}$ plays in the non-dependent one. 

In the third section, relying on the fibered presentation of operads à la Lurie, we propose a notion of dependent affine operad. 

In the fourth section, we study the basic theory of dependent operads over $\mathbf{Dep_{aff}}$, establishing notably that the category of algebras over such an operad is monadic over a presheaf topos that represents a category of "dependent sets", in an appropriate sense we introduce, based on the signature corresponding to a given operad.

In the fifth section, we aim to allow variable duplication by studying a monad on $\mathbf{FinPos}$ that plays the same role as the Leinster monad of a perfect operator category. In particular, we obtain an "inert/active" weak factorization of the corresponding Kleisli category.

In the last section, we recast the definitions and results of Sections 3 and 4 in this "general" setting, giving the important example of the operad for categories.

\section{The category $\mathbf{Dep_{aff}}$}

We first construct a category modeling dependent operations. Our objective, at first, is to be able to work in the framework of Barwick, introduced in \cite{barwick2018operator}.

\begin{definition}
 We define a category $\mathbf{Dep_{aff}}$ as follows:
 \begin{itemize}
  \item The objects are the finite posets.
  \item The morphisms $F \to F'$ are the opfibrations between such posets.
 \end{itemize}
 \end{definition}
 
 \begin{remark}
  We think of the posets as (generalized) contexts, their objects as types, and we think of the order as witnessing dependencies. As for the morphisms, for each type in the codomain and each dependency of such a type, a universal corresponding dependency should be provided for every preimage in the domain of the type we started with (this is why the notion of Grothendieck opfibration is relevant).

 \end{remark}
 
 \begin{proposition}
 \label{finitely_complete}
  $\mathbf{Dep_{aff}}$ is finitely complete.
 \end{proposition}
 
 \begin{proof}
  The one object poset is clearly a terminal object. We now turn to the existence of pullbacks. Given maps $f : J \to I$ and $g : K \to I$ in $\mathbf{Dep_{aff}}$, we will prove that the following pullback square computed in $\mathbf{Cat}$
\[\begin{tikzcd}[ampersand replacement=\&]
	L \&\& J \\
	\\
	K \&\& I
	\arrow["{g'}", from=1-1, to=1-3]
	\arrow["{f'}"', from=1-1, to=3-1]
	\arrow["\ulcorner"{anchor=center, pos=0.125}, draw=none, from=1-1, to=3-3]
	\arrow["f", from=1-3, to=3-3]
	\arrow["g"', from=3-1, to=3-3]
\end{tikzcd}\]
  defines a pullback square in $\mathbf{Dep_{aff}}$. $L$ clearly inherits the properties of being a finite poset. 

Moreover, $f'$ and $g'$ are opfibrations (as pullbacks of opfibrations). The square above hence lies in $\mathbf{Dep_{aff}}$. The only thing left to check is that it has the correct universal property in $\mathbf{Dep_{aff}}$. To see this, consider a cone
\[\begin{tikzcd}[ampersand replacement=\&]
	A \&\&\& \\
	\& L \&\& J \\
	\\
	\& K \&\& I
	\arrow["m", dashed, from=1-1, to=2-2]
	\arrow["{h_1}", curve={height=-12pt}, from=1-1, to=2-4]
	\arrow["{h_0}"', curve={height=12pt}, from=1-1, to=4-2]
	\arrow["{g'}", from=2-2, to=2-4]
	\arrow["{f'}"', from=2-2, to=4-2]
	\arrow["\ulcorner"{anchor=center, pos=0.125}, draw=none, from=2-2, to=4-4]
	\arrow["f", from=2-4, to=4-4]
	\arrow["g"', from=4-2, to=4-4]
\end{tikzcd}\]
where $h_0$ and $h_1$ are arrows in $\mathbf{Dep_{aff}}$. To see that $m$ is an opfibration, consider an arrow $u : m(a) \to l$ in $L$. We may factor $u$ as an $f$'-cocartesian arrow followed by an $f'$-vertical one $u'$. Further factoring $u'$ as a $g'$-cocartesian arrow followed by a $g'$-vertical one, the image of this latter composite through $f'$ is the identity, which implies that both arrows are  mapped to identities by $f'$ since there are no non-trivial retractions in the poset $K$. Overall, we have a factorization of $u$ as an $f'$-cocartesian arrow, followed by a $g'$-cocartesian one, followed by an arrow that is vertical both with respect to $f'$ and $g'$. The latter arrow must be an identity since the square we consider is a pullback square.

Therefore, we may treat separately the case where $u$ is $f'$-cocartesian and where it is $g'$-cocartesian. Assume that $u$ is $f'$-cocartesian, and consider a cocartesian lift $v : a \to b$ of $f'(u) : h_0(a) \to f'(l)$. Then, $m(v)$ is $f'$-cocartesian, so it coincides with $f'(u)$ since there are no non-trivial isomorphisms in the poset $L$. Moreover, $m(v)$ being $f'$-cocartesian and $v$ being $h_0$ cocartesian jointly imply that $v$ is $m$-cocartesian.
By symmetry, we can also take care of the case where $u$ is $g'$-cocartesian.

Overall, $m$ is an opfibration, which concludes the proof.
\end{proof}

\begin{remark}
 Since the pullback in $\mathbf{Dep_{aff}}$ are computed in $\mathbf{Cat}$, it follows that monomorphisms in $\mathbf{Dep_{aff}}$ are also monomorphisms in $\mathbf{Cat}$. Conversely, the opfibrations that are monomorphisms in $\mathbf{Cat}$ are monomorphisms in $\mathbf{Dep_{aff}}$. Since such functors must be surjective on morphisms, those correspond precisely to the functors that are fully faithful and injective on objects. 
\end{remark}

\begin{proposition}
\label{regular}
 $\mathbf{Dep_{aff}}$ is a regular category.
\end{proposition}

\begin{proof}
 Given a map $F : I \to J$ in $\mathbf{Dep_{aff}}$, we may factor through its image:

\[\begin{tikzcd}[ampersand replacement=\&]
	I \& {\mathbf{Im} F} \& J
	\arrow[from=1-1, to=1-2]
	\arrow["F", curve={height=-18pt}, from=1-1, to=1-3]
	\arrow[from=1-2, to=1-3]
\end{tikzcd}\]
The first component of the bottom composite is surjective on objects, and the second is a monomorphism (in $\mathbf{Dep_{aff}}$). It is not difficult to see that the first component is also an opfibration: to lift a map in $\mathbf{Im}F$, it is enough to consider it as a map in $J$ and use the fact that $F$ is an opfibration to obtain a cocartesian lift, which will be cocartesian since every map in $\mathbf{Im}F$ is cocartesian relative to the second inclusion.

We claim that the strong epimorphisms are the maps in $\mathbf{Dep_{aff}}$ that are surjective on objects. To see this, consider a lifting problem as in the diagram below,
\[\begin{tikzcd}[ampersand replacement=\&]
	A \&\&\& \\
	\& {B \times_J I} \&\& I \\
	\\
	\& B \&\& J
	\arrow[dashed, from=1-1, to=2-2]
	\arrow[curve={height=-12pt}, from=1-1, to=2-4]
	\arrow["P"', curve={height=12pt}, from=1-1, to=4-2]
	\arrow[from=2-2, to=2-4]
	\arrow[from=2-2, to=4-2]
	\arrow["\ulcorner"{anchor=center, pos=0.125}, draw=none, from=2-2, to=4-4]
	\arrow["F", from=2-4, to=4-4]
	\arrow[curve={height=-6pt}, dashed, from=4-2, to=2-2]
	\arrow["H"{description}, dashed, from=4-2, to=2-4]
	\arrow[from=4-2, to=4-4]
\end{tikzcd}\]
with $H$ a solution in $\mathbf{Cat}$ (since functors surjective on objects have the left lifting property against the fully faithful functors monic on objects). This gets us a section of the functor $B \times_J I \to B$, which is therefore an isomorphism since it is both fully faithful and monic on objects (as a pullback of $F$) and surjective on objects (since it has a section). This shows that $H$ is an opfibration as a pullback of the bottom opfibration: it is also a solution to the lifting problem in $\mathbf{Dep_{aff}}$. Therefore, a functor in $\mathbf{Dep_{aff}}$ that is surjective on objects has the left lifting property against the monomorphisms. 

For the converse, we use the usual retract argument. Consider a functor $P : A \to B$ having the left lifting property against the monomorphisms. Factoring $P$ through its image, we may consider the following lifting problem,
\[\begin{tikzcd}[ampersand replacement=\&]
	A \&\& {\mathbf{Im}P} \\
	\\
	B \&\& B
	\arrow[from=1-1, to=1-3]
	\arrow["P"', from=1-1, to=3-1]
	\arrow[from=1-3, to=3-3]
	\arrow["H"{description}, dashed, from=3-1, to=1-3]
	\arrow[from=3-1, to=3-3]
\end{tikzcd}\]
and a solution $H$ exhibiting $P$ as a retract of $A \to \mathbf{Im}F$. Therefore, $P$ is also surjective on objects (the class of such functors is retract stable since it is given by a right lifting property).

Since the strong epimorphisms are also stable under pullback, we can conclude: indeed, a result by Joyal \cite{252842} tells us that a finitely complete category with a pullback-stable factorization system (strong epimorphism, monomorphism) is a regular category. 
\end{proof}

We write $\mathbf{Dep_{aff}^+}$ for the wide subcategory of $\mathbf{Dep_{aff}}$ whose morphisms are the monomorphisms, and $\mathbf{Dep_{aff}^{elm}}$ for the full subcategory of the former spanned by the finite posets admitting an initial object.
For $I$ an object of $\mathbf{Dep_{aff}}$, we consider the diagram $D_I$ defined below:
\[\begin{tikzcd}[ampersand replacement=\&]
	{\iota \downarrow I} \&\& {\mathbf{Dep_{aff}^{elm}}} \\
	\\
	{*} \&\& {\mathbf{Dep_{aff}^+}}
	\arrow["{D_I}", from=1-1, to=1-3]
	\arrow[from=1-1, to=3-1]
	\arrow[between={0.3}{0.7}, Rightarrow, from=1-3, to=3-1]
	\arrow["{\iota }", from=1-3, to=3-3]
	\arrow["I", from=3-1, to=3-3]
\end{tikzcd}\]

\begin{proposition}
\label{elm_diag_depaff}
 The image $D_I'$ of the canonical cocone on $D_I$ with apex $I$ through the inclusion $\mathbf{Dep_{aff}^{elm}} \subset \mathbf{Dep_{aff}}$ is a colimiting cocone. 
\end{proposition}

\begin{proof}
Consider a cocone on $D_I'$ with apex $J$ and consider an object $i$ of $I$. The cosieve inclusion $I_{i/} \subset I$ is an object of the diagram, so we have a map $I_{i/} \to J$. We define $m(i)$ to be the image of the initial object of $I_{i/}$ under this map. This mapping clearly extends to a functor $m : I \to J$ that yields a factorization of the cocone through the canonical cocone on $D_I'$ with apex $I$. To show that $m$ is an opfibration, consider an object $i \in I$ and an arrow $m(i) \to j$ in $J$. Using that $I_{i/} \to J$ is an opfibration, we get a cocartesian lift $i \to i'$ in $I_{i/}$ whose image under the full inclusion $I_{i/} \subset I$ is therefore still cocartesian (with respect to $m$).
It is also clear that this factorization is the unique possible one.

This proves that $I$ is canonically the colimit of $D_I'$.
\end{proof}

\begin{remark}
 Note that $D_I'$ is also a colimit cocone in $\mathbf{FinPos}$.
\end{remark}

\section{Dependent operator category}
 
 In this section, we will show that $\mathbb{Dep_aff}$ is a perfect operator category in the sense of Barwick, and then observe that it moreover enjoys additional properties. 
 
 Recall the following definition:
 
 \begin{definition}[{\cite[Definition 1.2]{barwick2018operator}}]
  An \textbf{operator category} $\Phi$ is an essentially small category that satisfies the following three conditions:
  \begin{enumerate}
   \item The category $\Phi$ has a terminal object.
   \item For any morphism $J \to I$ of $\Phi$ and for any point $i \in |I| := Hom_\Phi(1,I)$, there exists a fiber $J_i := \{i\} \times_I J$.
   \item For any pair of objects $I,\;J \in \Phi$, the set $Hom_\Phi(J,I)$ is finite.
  \end{enumerate}  
 \end{definition}
 
 Our definition for $\mathbf{Dep_{aff}}$ falls into this framework:
 
  \begin{proposition}
 \label{dep_opc}
  $\mathbf{Dep_{aff}}$ is an operator category.
 \end{proposition}
 
 \begin{proof}
  $\mathbf{Dep_{aff}}$ is indeed locally finite. Moreover, since it is finitely complete by \cref{finitely_complete}, it admits in particular a terminal object and fibers.
 \end{proof}

 \begin{remark}
  In $\mathbf{Dep_{aff}}$, the points (i.e. the global sections) correspond to maximal objects of a given poset. This is enforced by the fact that such a global section $* \to I$ must, in particular, lift against the source inclusion $s : \mathbb{1} \to \mathbb{2}$.
 \end{remark}
 
 \begin{definition}[{\cite[Definition 4.1]{barwick2018operator}}]
  A point classifier in a operator category $\Phi$ is a terminal object in the category $\Phi^\text{cons}$ whose objects are the points in $\Phi$ and whose morphism are pullback squares.
 \end{definition}
 
 \begin{lemma}
 \label{point_classifier}
  $\mathbf{Dep_{aff}}$ admits a point classifier.
 \end{lemma}
 
 \begin{proof}
 The point classifier is the following poset,
\[\begin{tikzcd}[ampersand replacement=\&]
	\& \bullet \\
	\bullet \& {\bullet^*}
	\arrow[from=1-2, to=2-2]
\end{tikzcd}\]
where the special point is the indicated one.

Given a finite poset $F$ and a point $* \to F$ (i.e a maximal object), there is a unique pullback square
\[\begin{tikzcd}[ampersand replacement=\&]
	{*} \&\& F \\
	\\
	{*} \&\& T
	\arrow["p", from=1-1, to=1-3]
	\arrow[from=1-1, to=3-1]
	\arrow["\ulcorner"{anchor=center, pos=0.125}, draw=none, from=1-1, to=3-3]
	\arrow["{\chi_p}", from=1-3, to=3-3]
	\arrow["t"', from=3-1, to=3-3]
\end{tikzcd}\]
where $\chi_p : F \to T$ maps the given node to the special point, its predecessors to the predecessor of the special point, and the rest of $F$ to the remaining object (this definition is the only possible one to make $\chi_p$ a morphism in $\mathbf{Dep_{aff}}$).
 \end{proof}
 
\begin{remark}
\label{cosieve_classifier}
  Actually, $T$ is a cosieve classifier: it classifies any cosieve inclusion $S \to P$. Namely, write $\mathbf{Dep_{aff}}^\text{s}$ for the category whose objects are the cosieves $S \to F$ for $F$ in $\mathbf{Dep_{aff}}$, and whose morphisms are the pullbacks squares:

\[\begin{tikzcd}[ampersand replacement=\&]
	S \&\& F \\
	\\
	{S'} \&\& {F'}
	\arrow[from=1-1, to=1-3]
	\arrow[from=1-1, to=3-1]
	\arrow["\ulcorner"{anchor=center, pos=0.125}, draw=none, from=1-1, to=3-3]
	\arrow[from=1-3, to=3-3]
	\arrow[from=3-1, to=3-3]
\end{tikzcd}\]
  Then $* \to T$ is a terminal object of $\mathbf{Dep_{aff}}^\text{s}$.
\end{remark}

\begin{definition}{\cite[Definitions 4.3 and 4.6]{barwick2018operator}}]
 Suppose $\Phi$ is an operator category, and suppose $(T, t)$ a point classifier for $\Phi$. Then for any object $(I, i)$ of $\Phi_*$, the unique conservative morphism
$(I, i) \to (T, t)$ will be called the classifying morphism for $i$, and will be denoted
$\chi_i$. We shall call the point $t \in |T|$ the special point of $T$, and for any morphism
$I \to T$ of $\Phi$, the fiber $I_t$ will be called the special fiber . Write $\text{fib} : \Phi_{/T}  \to \Phi$ for the special fiber functor $I \mapsto I_t$.

An operator category $\Phi$ is \textbf{perfect} if the following conditions are satisfied.
\begin{enumerate}
 \item The category $\Phi_*$ contains a point classifier $(T_\Phi,t_\Phi)$.
 \item The special fiber functor $\text{fib}: \Phi_{/T_\Phi} \to \Phi$ admits a right adjoint $E_\Phi$.
\end{enumerate}

\end{definition}
  
 \begin{proposition}
  $\mathbf{Dep_{aff}}$ is a perfect operator category.
 \end{proposition}
 
 \begin{proof}
  By \cref{point_classifier}, $\mathbf{Dep_{aff}}$ admits a point classifier $t : * \to T$. We will check that the special fiber functor $(-)_s : \mathbf{Dep_{aff}}_{/T} \to \mathbf{Dep_{aff}}$ admits a right adjoint. Given a finite poset $I$, write $E(I)$ for the following pushout (computed in $\mathbf{Cat}$):

\[\begin{tikzcd}[ampersand replacement=\&]
	{\text{Ob}(I)} \&\& {* \amalg (\text{Ob}(I) \times \rightarrow)} \\
	\\
	I \&\& {E(I)}
	\arrow["{! \amalg \text{Ob}(I) \times t}", from=1-1, to=1-3]
	\arrow[from=1-1, to=3-1]
	\arrow[from=1-3, to=3-3]
	\arrow[from=3-1, to=3-3]
	\arrow["\ulcorner"{anchor=center, pos=0.125, rotate=180}, draw=none, from=3-3, to=1-1]
\end{tikzcd}\]
Concretely, this adds a predecessor to each object of the poset as well as an isolated object.
  This construction is functorial, and we always have a pullback square

\[\begin{tikzcd}[ampersand replacement=\&]
	I \&\& {E(I)} \\
	\\
	{*} \&\& T
	\arrow[from=1-1, to=1-3]
	\arrow[from=1-1, to=3-1]
	\arrow["\ulcorner"{anchor=center, pos=0.125}, draw=none, from=1-1, to=3-3]
	\arrow["{e_I}", from=1-3, to=3-3]
	\arrow["t"', from=3-1, to=3-3]
\end{tikzcd}\]
where $e_I$ maps the "copy" of $I$ to the special point, the added predecessors to the only predecessor of the special point in $T$, and the isolated object to the isolated object of $T$.
We therefore have a functor $E : \mathbf{Dep_{aff}} \to  \mathbf{Dep_{aff}}_{/T}$. The unit is obtained as follows.

Starting from a map $J \to T$ and the pullback on left below, we also have the pullback on the right by definition of $E(J_s)$.
\[\begin{tikzcd}[ampersand replacement=\&]
	{J_s} \&\& J \&\& {J_s} \&\& {E(J_s)} \\
	\\
	{*} \&\& T \&\& {*} \&\& T
	\arrow[from=1-1, to=1-3]
	\arrow[from=1-1, to=3-1]
	\arrow["\ulcorner"{anchor=center, pos=0.125}, draw=none, from=1-1, to=3-3]
	\arrow["s", from=1-3, to=3-3]
	\arrow["{f_{J_s}}", from=1-5, to=1-7]
	\arrow[from=1-5, to=3-5]
	\arrow["\ulcorner"{anchor=center, pos=0.125}, draw=none, from=1-5, to=3-7]
	\arrow["{e_{J_s}}", from=1-7, to=3-7]
	\arrow["t"', from=3-1, to=3-3]
	\arrow["t"', from=3-5, to=3-7]
\end{tikzcd}\]
Now the former factors through the latter: we can define a map $J \to E(J_s)$ by mapping the subposet $J_s$ of $J$ to the corresponding subposet of $E(J_s)$, and the remaining objects $j \in J$ to either a suitable "new" predecessor or the isolated object depending on whether there exists an object $i > j$ with $i \in J_s$ or not (i.e., depending on whether the image of $j$ under $s : J \to T$ is the predecessor of the special point or is the isolated one). In the first case, we take the image of $j$ to be the "new" predecessor of the (target of the) cocartesian lift of the arrow to the special point and from its predecessor (which we think of as the image of $j$ under $s$). This yields the unit $J \to E(J_s)$ above $T$.

The counit has components identity morphisms (by definition of the morphism $E(I) \to T$, its special fiber is just $I$). Therefore, to check the triangular identities, we just have to check that $\eta_{E(I)}$ and $(\eta_{J})_s$ are the identity morphisms. The case of $\eta_{E(I)}$ is evident by construction of the unit above.

For $(\eta_{J})_s$, computing the special fiber for $J$ and $E(J_s)$ yields $J_s$ in both cases (see the two pullbacks above), and the induced map $J_s \to J_s$ is clearly the identity.

Overall, we can conclude that $E : \mathbf{Dep_{aff}} \to  \mathbf{Dep_{aff}}_{/T}$ indeed is right adjoint to the special fiber functor.
\end{proof}

We may propose the following definition to take into account the additional properties enjoyed by $\mathbf{Dep_{aff}}$:

\begin{definition}
 A \textbf{dependent} operator category $\Phi$ is an operator category that is also regular and comes equipped with a subset of elementary objects $\mathcal{E}$ such that:
 \begin{itemize}
  \item The full subcategories $\mathcal{E}_\Phi(X) \subset \Phi_{/X}^{op}$ spanned by the monomorphisms with domain an elementary object for $X \in \Phi$, are finite.
  \item If $X$ is elementary, then $X \to *$ is a strong epimorphism.
  \item The full subcategory of $\Phi$ spanned by the elementary objects is closed under pullback along strong epimorphisms.  
  \item If $X \to Y$ is a strong epimorphism and $X$ is elementary, then so is $Y$.
 \end{itemize}
 An admissible morphism between dependent operator category is a regular functor preserving the elementary objects.
\end{definition}

Of course, our canonical example is the following.

\begin{proposition}
 $\mathbf{Dep_{aff}}$ is a dependent operator category. 
\end{proposition}

\begin{proof}
 We already know $\mathbf{Dep_{aff}}$ to be an operator category. We take the elementary objects to be the finite posets that admit an initial object.
 
  If $X \to Y$ is a strong epi in $\mathbf{Dep_{aff}}$, i.e. a functor surjective on objects, and if $X$ is elementary, then the image of the initial object of $X$ under $X \to Y$ is clearly initial in $Y$.
 
 Consider a pullback 

\[\begin{tikzcd}[ampersand replacement=\&]
	{I \times_K J} \&\& J \\
	\\
	I \&\& K
	\arrow[from=1-1, to=1-3]
	\arrow[two heads, from=1-1, to=3-1]
	\arrow["\ulcorner"{anchor=center, pos=0.125}, draw=none, from=1-1, to=3-3]
	\arrow[two heads, from=1-3, to=3-3]
	\arrow[from=3-1, to=3-3]
\end{tikzcd}\]
where $I$, $J$ and $K$ are elementary and $J \to K$ is a strong epimorphism. Since we can factor $I \to K$ as a strong epimorphism followed by a monomorphism, it will be enough to prove that $I \times_K J$ is elementary with any of these two assumption on $I \to K$.
If $I \to K$ is a strong epimorphism, then we are considering a pullback where both functors preserve the initial objects (cf. the observation above that strong epimorphisms preserve initial objects), so the pullback has an initial object (which is given by taking the componentwise initial object of the $I \times_K J$).

If $m : I \to K$ is a monomorphism, we may consider the initial object $i$ of $I$ and a cartesian lifting $j$ of $k \to m(i)$, where $k$ is the initial object of $K$. This defines an object of $I \times_K J$ that we claim to be initial. To see this, consider any other object of $I \times_K J$ with components $(i',j')$, we have a commutative triangle as below,
\[\begin{tikzcd}[ampersand replacement=\&]
	\& j \& \\
	{j_0} \&\& {j'} \\
	\& {m(i)} \\
	k \&\& {m(i')}
	\arrow[dashed, from=1-2, to=2-3]
	\arrow[dotted, no head, from=1-2, to=3-2]
	\arrow[from=2-1, to=1-2]
	\arrow[from=2-1, to=2-3]
	\arrow[dotted, no head, from=2-1, to=4-1]
	\arrow[dotted, no head, from=2-3, to=4-3]
	\arrow[from=3-2, to=4-3]
	\arrow[from=4-1, to=3-2]
	\arrow[from=4-1, to=4-3]
\end{tikzcd}\]
where $j_0$ is the initial object of $J$, so we deduce the existence of a (unique) map $j \to j'$ over $m(i) \to m(i')$ since $j_0 \to j$ is cocartesian. This provides the map $(i,j) \to (i',j')$ in $I \times_K J$, proving that $(i,j)$ defines an initial object.

\end{proof}

\begin{example}
 $\mathbf{FinSet}$ can be made a dependent operator category by taking every object to be elementary.
\end{example}

For every dependent operator category $\Phi$, there is a canonical mapping at the level of objects $|-| : \mathbf{Ob}(\Phi) \to \mathbf{Cat}$ defined by:
$$X \mapsto \mathcal{E}_\Phi(X)$$

\begin{lemma}
 The mapping above extends to a regular functor $|-| : \Phi \to \mathbf{Dep_{aff}}$.
\end{lemma}

\begin{proof}
 First, note that it indeed takes values in finite posets since $\mathcal{E}_\Phi(X)$ is finite by assumption and since there is at most one morphism between two objects (this is because the objects are given by monomorphisms).
 
  If $p : X \to Y$ is a morphism in $\Phi$, given an object $A \to X$ of $\mathcal{E}_\Phi(X)$, then form the factorization as in the following diagram: 

\[\begin{tikzcd}[ampersand replacement=\&]
	A \&\& X \\
	\\
	B \&\& Y
	\arrow[hook, from=1-1, to=1-3]
	\arrow[two heads, from=1-1, to=3-1]
	\arrow[from=1-3, to=3-3]
	\arrow[hook, from=3-1, to=3-3]
\end{tikzcd}\]
and we map $A \to X$ to the monomorphism $B \to Y$, which is an object of $\mathcal{E}_\Phi(Y)$ since $B$ is elementary because it admits a strong epi from the elementary object $A$. Also, observe that given a monomorphism $m : X \to Y$, the action boils down to postcomposition with $m$.

This clearly extends the mapping $|-|$ to a functor $\Phi \to \mathbf{FinPos}$. Let us check that the action of a morphism $X \to Y$ yields an opfibration. For this, consider an arrow $B \leftarrow B'$ in $\mathcal{E}_\Phi(Y)$ such that $B \to Y$ is the image of $A \to X$ in $\mathcal{E}_\Phi(X)$. We form the pullback as in the following diagram,
\[\begin{tikzcd}[ampersand replacement=\&]
	\& A \&\& X \\
	{A'} \\
	\& B \&\& Y \\
	{B'}
	\arrow[hook, from=1-2, to=1-4]
	\arrow[two heads, from=1-2, to=3-2]
	\arrow[from=1-4, to=3-4]
	\arrow[hook, from=2-1, to=1-2]
	\arrow["\ulcorner"{anchor=center, pos=0.125}, draw=none, from=2-1, to=3-2]
	\arrow[two heads, from=2-1, to=4-1]
	\arrow[hook, from=3-2, to=3-4]
	\arrow[hook, from=4-1, to=3-2]
	\arrow[from=4-1, to=3-4]
\end{tikzcd}\]
where the indicated maps are strong epimorphisms or monomorphisms by stability under pullback, proving an object $A' \to X$ of $\mathcal{E}_\Phi(X)$ together with a map $A \leftarrow A'$ that is clearly a cocartesian lifting of $B \leftarrow B'$ thanks to the universal property of the pullback. This proves that $\mathcal{E}_\Phi(X) \to \mathcal{E}_\Phi(Y)$ is an opfibration, as required.

We still need to establish that the functor $|-|$ is regular.
To check preservation of finite limits, it is enough to prove that the terminal object, binary products, and pullbacks along monomorphisms are preserved since all finite limits can be computed from these.

It is clear that the terminal object is mapped to the terminal poset.

Consider a binary product $X \times Y$. We claim that the canonical map $m : \mathcal{E}_\Phi(X \times Y) \to \mathcal{E}_\Phi(X) \times \mathcal{E}_\Phi(Y)$ is an isomorphism. First, it is surjective on objects: Given $A \to X$ and $B \to Y$ two monomorphisms from elementary objects, $A \times B \to X \times Y$ is a monomorphism from the object $A \times B$ (which is elementary because $A$ and $B$ are) fitting in a diagram
\[\begin{tikzcd}[ampersand replacement=\&]
	{A \times B} \&\& {X \times Y} \\
	\\
	A \&\& X
	\arrow[hook, from=1-1, to=1-3]
	\arrow[two heads, from=1-1, to=3-1]
	\arrow[from=1-3, to=3-3]
	\arrow[hook, from=3-1, to=3-3]
\end{tikzcd}\]
where $A \times B \to A$ is a strong epimorphism as the pullback of $B \to *$ (which is a strong epi because $B$ is elementary), and similarly for $A \times B \to A$. This shows that $m$ is surjective on objects. Likewise, it is also full. Since it is a functor between posets, it is automatically faithful, and $m$ is therefore an isomorphism.

Next, consider a pullback along a monomorphism $X \to Z$, as below left, and the induced diagram on right below:
\[\begin{tikzcd}[ampersand replacement=\&]
	\&\&\& {\mathcal{E}_\Phi(X \times_Z Y)} \&\&\& \\
	{X \times_Z Y} \&\& X \&\& {\mathcal{E}_\Phi(X) \times_{\mathcal{E}_\Phi(Z)} \mathcal{E}_\Phi(Y)} \&\& {\mathcal{E}_\Phi(X)} \\
	\\
	Y \&\& Z \&\& {\mathcal{E}_\Phi(Y)} \&\& {\mathcal{E}_\Phi(Z)}
	\arrow["m", dashed, from=1-4, to=2-5]
	\arrow[from=2-1, to=2-3]
	\arrow[from=2-1, to=4-1]
	\arrow["\ulcorner"{anchor=center, pos=0.125}, draw=none, from=2-1, to=4-3]
	\arrow[from=2-3, to=4-3]
	\arrow[from=2-5, to=2-7]
	\arrow[from=2-5, to=4-5]
	\arrow[from=2-7, to=4-7]
	\arrow[from=4-1, to=4-3]
	\arrow[from=4-5, to=4-7]
\end{tikzcd}\]

Consider an object of $\mathcal{E}_\Phi(X) \times_{\mathcal{E}_\Phi(Z)} \mathcal{E}_\Phi(Y)$; it corresponds to a diagram as in the left below,
\[\begin{tikzcd}[ampersand replacement=\&]
	A \&\& Y \&\& A \&\&\& {X \times_Z Y} \\
	\& X \&\&\&\&\& X \\
	B \&\& Z \&\& B \&\&\& Z
	\arrow[hook, from=1-1, to=1-3]
	\arrow[two heads, from=1-1, to=3-1]
	\arrow[from=1-3, to=3-3]
	\arrow[hook, from=1-5, to=1-8]
	\arrow[two heads, from=1-5, to=3-5]
	\arrow[from=1-8, to=2-7]
	\arrow[from=1-8, to=3-8]
	\arrow[hook, from=2-2, to=3-3]
	\arrow[hook, from=2-7, to=3-8]
	\arrow[hook, from=3-1, to=2-2]
	\arrow[hook, from=3-1, to=3-3]
	\arrow[curve={height=-18pt}, from=3-5, to=1-8]
	\arrow[hook, from=3-5, to=2-7]
	\arrow[hook, from=3-5, to=3-8]
\end{tikzcd}\]
and this factors through the pullback with the map $A \to X \times_Z Y$ being a monomorphism since $A \to Y$ is a monomorphism by assumption. Because the strong-epi/mono factorization is orthogonal, the factorization we get is unique, and we still get $B \to X$ (resp. $B \to Z$) when evaluating how $X \times_Z Y \to X$ (resp. $X \times_Z Y \to Z$) acts on $A \to X \times_Z Y$.

This provides the preimage we were seeking, proving that $m$ is surjective on objects. Likewise, it is also full. As before, this is enough to establish that $m$ is an isomorphism.

Finally, we have to prove that strong epimorphisms are mapped to strong epimorphisms. To this end, consider a strong epimorphism $X \to Y$ in $\Phi$. We must prove that $\mathcal{E}_\Phi(X) \to \mathcal{E}_\Phi(Y)$ is a strong epimorphism in $\mathbf{Dep_{aff}}$, i.e., that is surjective on objects. Consider an object $B \to Y$ of $\mathcal{E}_\Phi(Y)$ and form the following pullback:
\[\begin{tikzcd}[ampersand replacement=\&]
	A \&\& B \\
	\\
	X \&\& Y
	\arrow[two heads, from=1-1, to=1-3]
	\arrow[hook, from=1-1, to=3-1]
	\arrow["\ulcorner"{anchor=center, pos=0.125}, draw=none, from=1-1, to=3-3]
	\arrow[hook, from=1-3, to=3-3]
	\arrow[two heads, from=3-1, to=3-3]
\end{tikzcd}\]
It appears that $A \to X$ then defines a preimage of $B \to Y$ through the functor $\mathcal{E}_\Phi(X) \to \mathcal{E}_\Phi(Y)$. This concludes the proof.
\end{proof}

\begin{proposition}
 The functor $|-| : \Phi \to \mathbf{Dep_{aff}}$ is admissible.
\end{proposition}

\begin{proof}
 We have already seen in \cref{regular} that it is regular. If $X$ is elementary, then the identity $X \to X$ is a monomorphism from an elementary object that defines an initial object in $\mathcal{E}_\Phi(X)$. Therefore, $\mathcal{E}_\Phi(X)$ is elementary by definition of the dependent operator category structure on $\mathbf{Dep_{aff}}$.
\end{proof}

\begin{definition}
 An admissible morphism $F : \Phi \to \Psi$ such that the induced morphism $|X| \to |FX|$ is surjective on objects will be called an \textbf{operator morphism}.
\end{definition}

\begin{definition}
 A dependent operator category $\Phi$ is perfect when it is perfect as an operator category, and when moreover its point classifier is a subobject classifier.
\end{definition}

\begin{example}
 By \cref{cosieve_classifier}, $\mathbf{Dep_{aff}}$ is perfect as a dependent operator category (the monomorphisms are the cosieve inclusions).
\end{example}

The following definition will be useful for defining our notion of dependent operad:

\begin{definition}
 For $\Phi$ a dependent operator category, and $I$ an object of $\Phi$, we write $\mathbf{ElD}I$ for the diagram in $\Phi$ given by the monomorphisms $X \to I$ with $X$ elementary, and the commutative triangles involving two such maps. We call this diagram the elementary diagram for $I$.
\end{definition}

\begin{lemma}
\label{pb_cls_inert}
 If $I \to J$ is a monomorphism in a perfect dependent operator category $\Phi$, then it canonically fits in a pullback square as follows:

\[\begin{tikzcd}[ampersand replacement=\&]
	I \&\& J \\
	\\
	I \&\& {E(I)}
	\arrow[from=1-1, to=1-3]
	\arrow["id"', from=1-1, to=3-1]
	\arrow["\ulcorner"{anchor=center, pos=0.125}, draw=none, from=1-1, to=3-3]
	\arrow[from=1-3, to=3-3]
	\arrow["{\eta_I}"', from=3-1, to=3-3]
\end{tikzcd}\]
\end{lemma}

\begin{proof}
 Since $I$ is obtained from $E(I) \to T$ by applying the special fiber functor $(-)_s$, the monomorphism $I \to J$, which is classified by some map $\chi : J \to T$, the unit of adjunction yields a map $J \to E(I)$ above $T$. Now, in the following diagram, the bottom square and the full rectangle are pullbacks,

\[\begin{tikzcd}[ampersand replacement=\&]
	I \&\& J \\
	\\
	I \&\& {E(I)} \\
	\\
	{*} \&\& T
	\arrow[from=1-1, to=1-3]
	\arrow["id"', from=1-1, to=3-1]
	\arrow["\ulcorner"{anchor=center, pos=0.125}, draw=none, from=1-1, to=3-3]
	\arrow["{\eta_J}", from=1-3, to=3-3]
	\arrow["{\eta_I}"', from=3-1, to=3-3]
	\arrow[from=3-1, to=5-1]
	\arrow["\ulcorner"{anchor=center, pos=0.125}, draw=none, from=3-1, to=5-3]
	\arrow["{e_I}", from=3-3, to=5-3]
	\arrow["t"', from=5-1, to=5-3]
\end{tikzcd}\]
 therefore, so is the top square, as we wanted to establish.
\end{proof}

Recall the following definition

\begin{definition}[{\cite[Definition 7.1]{barwick2018operator}}]
 Suppose that $\Phi$ is a perfect operator category. We call a morphism $J \to I$ of $\Lambda(\Phi)$ \textbf{inert} if, the corresponding morphism $J \to TI$ in $\Phi$ has the property that the natural morphism $$J \times_{TI} I \to I$$ is an isomorphism.
 Denote by $\Lambda(\Phi)^\dagger$ the collection of inert morphisms in $\Lambda(\Phi)$.
 
 Let us call a morphism $J \to I$ of $\Lambda(\Phi)$ \textbf{active} if the corresponding morphism $J \to TI$ of $\Phi$ factors as $$J \to TJ \to TI$$ where the morphism $TJ \to TI$ is of the form $T\phi$ for some morphism $\phi$ of $\Phi$. Denote by $\Lambda(\Phi)_\dagger$ the collection of active morphisms in $\Lambda(\Phi)$.
\end{definition}

 \cref{pb_cls_inert} therefore has the following consequence:
 
\begin{corollary}
\label{inert_duality}
For $\Phi$ a perfect dependent operator category, $\Lambda(\Phi)^\dagger$ is canonically isomorphic to $(\Phi^+)^{op}$, the opposite of the wide subcategory of $\Phi$ whose arrows are the monomorphisms.
\end{corollary}

\section{Dependent operads: the affine case}

Recall the following definition for non-dependent case:

\begin{definition}
 Suppose $\Phi$ is a perfect operator category. Then a $\Phi$-quasioperad or $\infty$-operad over $\Phi$ is an inner fibration $p : X^\otimes \to N\Lambda(\Phi)$ satisfying the following conditions.
 \begin{itemize}
  \item For every morphism $\phi : I \to J$ of $\Lambda(\Phi)^\dagger$ and every object $x \in X^\otimes_J$, there is a $p$-cocartesian edge $x \to y$ in $X^\otimes$ covering $\phi$.
  \item For any objects $I,J \in \Phi$, any object $x \in X^\otimes_I$ and $y \in X^\otimes_J$, any morphism $J \to I$ in $\Lambda(\Phi)$ and any $p$-cocartesian edges $\{y \to y_i| i \in |I| \}$ lying over the inert morphism $\{\rho_i : I \to \{i\} | i \in |I|\}$, the induced map
  $$Hom^\phi_{X^\otimes}(x,y) \to \Pi_{i \in |I|} Hom^{\rho_i \circ \phi}_{X^\otimes}(x,y_i)$$
  is an equivalence.
  \item For any object $I \in \Phi$, the $p$-cocartesian morphisms $\rho_i$ lying over the inert morphisms $\{I \to \{i\} | i \in |I|\}$ together induce an equivalence
  $$X^\otimes_I \to \Pi_{i \in |I|} X^\otimes_{\{i\}}$$
 \end{itemize}
 
\end{definition}

The key insight is that, in an operator category, we expect the object to behave as "products" of their points and require some kind of preservation of the colimits regarding the functors defining an operad. In the dependent case, the corresponding requirement involves elementary diagrams, having \cref{elm_diag_depaff} in mind for the case of $\mathbf{Dep_{aff}}$.

\begin{definition}
 Given a perfect dependent operator category $\Phi$, a dependent $(\infty-)$operad over $\Phi$ (or $\Phi$-quasioperad) is an inner fibration $O^\otimes \to \mathbf{N}\Lambda(\Phi)$ satisfying the following conditions:
 \begin{itemize}
  \item For every inert morphism $\phi : I \to J$ in $\Lambda(\Phi)$, and every object $x$ of $O^\otimes$ lying over $I$, there exists a cocartesian lift $x \to x'$ of $\phi$.
  \item For any objects $I$ and $J$ of $\Lambda(\Phi)$ and for any map $\phi : I \to J$, any objects $x \in O^\otimes_I$ and $y \in O^\otimes_J$ and any family of cocartesian edges $\{y \to y_i \;|\; i \in \mathbf{ElD}J\}$ lying over the morphisms in $\mathbf{ElD}J$, the induced map
  $$Hom^\phi(x,y) \to \varprojlim_{\rho_i \in \mathbf{ElD}J} Hom(x,y_i)^{\rho_i \circ \phi}$$
  is an equivalence.
  \item For any object $I \in \Lambda(\Phi)$ and the family of cocartesian edges $\{x \to x_k | x \in O^\otimes_I, k \in \mathbf{ElD}I\}$ lying over the morphisms $\rho_k : I \to I_k$, obtained from transposing monomorphisms $I_k \to I$ in $\mathbf{ElD}I$, is such that the induced map
  $$O^\otimes_I \to \varprojlim_{\rho_k \in \mathbf{ElD}I} O^\otimes_{I_k}$$
  is an equivalence.
 \end{itemize}
 
 We write $\mathbf{Op}_\Phi$ for the category of dependent operads over $\Phi$, where the morphisms are the functors $O_0^\otimes \to O_1^\otimes$ above $\mathbf{N}\Lambda(\Phi)$ that preserve cocartesian morphisms lying above an inert morphism.
\end{definition}

\begin{remark}
 The definition above is phrased in the quasicategorical setting to stress that the framework is naturally able to formulate notions of $\infty$-operads. However, in the present document, we will only consider $1$-categorical operads (i.e., the previous definition is applied to the nerve of a ($1$-)functor $O^\otimes \to \Lambda(\Phi)$.
\end{remark}

 A dependent $(\infty-)$operad over $\mathbf{Dep_{aff}}$ we will simply call a dependent $(\infty-)$operad.

\begin{definition}
 Given a perfect dependent operator category $\Phi$, a dependent signature over $\Phi$ is a cocartesian fibration $S^\otimes \to \mathbf{N}\Lambda(\Phi)^\dagger$ satisfying the following conditions:
 \begin{itemize}
  \item For any objects $I$ and $J$ of $\Lambda(\Phi)^\dagger$ and any map $\phi : I \to J$, any objects $x \in S^\otimes_I$ and $y \in S^\otimes_J$ and any family of cocartesian edges $\{y \to y_i \;|\; i \in \mathbf{ElD}J\}$ lying over the morphisms in $\mathbf{ElD}J$, the induced map
  $$Hom^\phi(x,y) \to \varprojlim_{\rho_i \in \mathbf{ElD}J} Hom(x,y_i)^{\rho_i \circ \phi}$$
  is an equivalence.
  \item For any object $I \in \Lambda(\Phi)^\dagger$, the family of cocartesian edges $\{x \to x_k | x \in S^\otimes_I, k \in \mathbf{ElD}I\}$ lying over the morphisms $\rho_k : I \to I_k$, obtained from transposing monomorphisms $I_k \to I$ in $\mathbf{ElD}I$, are such that the induced map
  $$S^\otimes_I \to \varprojlim_{\rho_k \in \mathbf{ElD}I} S^\otimes_{I_k}$$
  is an equivalence.
 \end{itemize}
 
 We write $\mathbf{Sig}_\Phi$ for the category of dependent signatures over $\Phi$, where the morphisms are the functors $S_0^\otimes \to S_1^\otimes$ above $\mathbf{N}\Lambda(\Phi)^\dagger$ that preserve cocartesian morphisms.
\end{definition}

There is a functor $U_\Phi : \mathbf{Op}_\Phi \to \mathbf{Sig}_\Phi$ defined by pulling back along the inclusion $\Lambda(\Phi)^\dagger \to \Lambda(\Phi)$.

\begin{definition}

 We write $\Lambda(\Phi)_0$ for the full subcategory of $\Lambda(\Phi)$ spanned by the elementary objects.
\end{definition}

\begin{lemma}
\label{ran_dep_sig_0}
 Let $\Lambda$ be a dependent operator category, and $P_0 : S^\otimes_0 \to \Lambda(\Phi)^\dagger_0$ an opfibration.
 
 Then $P_0$ extends to a dependent signature $P : S^\otimes \to \Lambda(\Phi)^\dagger$ that is unique up to equivalence.
\end{lemma}

\begin{proof}
 It follows from the definition of a dependent signature that the right Kan extension of a pseudo-functor $\Lambda(\Phi)^\dagger \to \mathbf{Cat}$ corresponding to the opfibration $P_0 : S^\otimes_0 \to \Lambda(\Phi)^\dagger_0$ along the inclusion $\Lambda(\Phi)^\dagger_0 \to \Lambda(\Phi)^\dagger$ yields a pseudo-functor corresponding to an opfibration $P : S^\otimes \to \Lambda(\Phi)^\dagger$ that satisfies the axioms for a dependent signature. Conversely, given such a dependent signature, the axioms imply that it must be equivalent to the opfibration obtained from $P_0$ by right Kan extension.
\end{proof}

\begin{lemma}
\label{ran_dep_op_0}
 Let $\Phi$ be a dependent operator category, and $P_0 : O^\otimes_0 \to \Lambda(\Phi)_0$ a functor satisfying the first condition for dependent operads, namely, such that for every inert morphism $\phi : F \to F'$ in $\Lambda(\Phi)_0$, and every object $x'$ of $O^\otimes_0$ lying over $F'$, there exists a cocartesian lift $x  \to x'$ of $\phi$.
 
 Then $P_0$ extends to a dependent operad $P : O^\otimes \to \Lambda(\Phi)$ that is unique up to equivalence.
\end{lemma}

\begin{proof}
 Using \cref{ran_dep_sig_0}, we first extend the "restriction" of $P_0$ to $\Lambda(\Phi)^\dagger_0$ to a dependent signature $S^\otimes \to \Lambda(\Phi)^\dagger$.
 Following the axioms for a dependent signature, we extend $O^\otimes_0$ for non-elementary objects and morphisms lying over any arrow of $\Lambda(\Phi)$ (not just those between elementary objects) in the unique way as to get a dependent operad. This extension is therefore unique up to equivalence.
\end{proof}

\section{Dependent operads over $\mathbf{Dep_{aff}}$}

In this section, we restrict our attention to the dependent operator category $\mathbf{Dep_{aff}}$.

\begin{example}
 Here, our running example will be that of graded monoids, that we construct by making use of  \cref{ran_dep_op_0}: we define $$(S_\mathbf{GMon}^\otimes)_0 \to \Lambda(\mathbf{Dep_{aff}})^\dagger_0$$ as the fibration that has exactly two non-empty fibers, which are isomorphic to the one point set, over the one object poset $\mathbb{1}$ and over the walking arrow $\mathbb{2}$. We denote the corresponding objects of $S_\mathbf{GMon}$ respectively by $b$ (that we think of as the base monoid) and $e$ (that we think of as the "total space" of monoids graded over the base). 
 
 The dependent operad for categories $$T_\mathbf{GMon} : O_\mathbf{GMon} \to \Lambda(\mathbf{Dep_{aff}})$$ has the same objects as $S_\mathbf{GMon}$ and is generated by two additional morphisms:
 \begin{itemize}
  \item $id_b : * \to b$ lying over the inclusion $\emptyset \to \mathbb{1}$.
  \item $comp_b : b \amalg b \to b$, where $b \amalg b$ is the (unique) object above the poset $\mathbb{1} \amalg \mathbb{1}$, lying over the unique functor $\mathbb{1} \amalg \mathbb{1} \to \mathbb{1}$.
  \item $id_e : b \to e$ lying over the domain inclusion $\mathbb{1} \to \mathbb{2}$.
  \item $comp_e : e \amalg_b e \to e$, where $b \amalg_b e$ is the (unique) object above the poset the span poset $\mathbb{2} \amalg_\mathbb{1} \mathbb{2}$, lying over the unique functor $\mathbb{2} \amalg_\mathbb{1} \mathbb{2} \to \mathbb{2}$ that preserves the domain and sends both codomains of legs to the codomain of the walking arrow $\mathbb{2}$.
 \end{itemize}
 These morphisms are subject to the usual equations for graded monoids.
\end{example}

\begin{definition}
Define a category $\mathbf{DepSet_{aff}}$ as follows:
\begin{itemize}
 \item The objects are pairs $(I, X)$ where $I \in \mathbf{Dep_{aff}}$ and $X \in \mathbf{Set}^{I}$.
 \item The morphisms from $(I, X)$ to $(J, Y)$ are pairs $(k,F)$ where $k : I \to TJ$ is a morphism from $I$ to $J$ in $\Lambda(\mathbf{Dep_{aff}})$ and $F \in Hom_{\mathbf{Set}^{J}}(\eta_J^* k_* X,Y)$, where $\eta_J : J \to TJ$ corresponds to the identity map on $J$ in $\Lambda(\mathbf{Dep_{aff}})$.
\end{itemize}
Composition is defined as follows:
\begin{itemize}
 \item The composition of the first component is that of $\Lambda(\mathbf{Dep_{aff}})$.
\item Given $(k_0,F_0) : (I,X) \to (J,Y)$ and $(k_1,F_1) : (J,Y) \to (K,Z)$, define the composite of $F_0$ and $F_1$ to be the image of $(F_0,F_1)$ under the map
\begin{align*}
& Hom_{\mathbf{Set}^{J}}(\eta_J^* (k_0)_* X,Y) \times Hom_{\mathbf{Set}^{K}}(\eta_K^* (k_1)_* Y,Z)\\
& \simeq Hom_{\mathbf{Set}^{TJ}}((k_0)_* X,(\eta_J)_*Y) \times Hom_{\mathbf{Set}^{TK}}((k_1)_* Y,(\eta_K)_*Z)\\
& \to  Hom_{\mathbf{Set}^{T^2K}}((Tk_1)_*(k_0)_* X,(Tk_1)_*(\eta_J)_*Y) \times Hom_{\mathbf{Set}^{TK}}(((k_1)_* Y, (\eta_K)_*Z)\\
& \to  Hom_{\mathbf{Set}{TK}}((\mu_{K})_*(Tk_1)_*(k_0)_* X,(\mu_{K})_*(Tk_1)_*(\eta_J)_*Y) \times Hom_{\mathbf{Set}^{TK}}((k_1)_* Y, (\eta_K)_*Z)\\
& \simeq  Hom_{\mathbf{Set}{TK}}((\mu_{K})_*(Tk_1)_*(k_0)_* X,(k_1)_* Y) \times Hom_{\mathbf{Set}^{TK}}((k_1)_* Y, (\eta_K)_*Z)\\
& \to  Hom_{\mathbf{Set}{TK}}((\mu_{K})_*(Tk_1)_*(k_0)_* X,(\eta_K)_*Z)\\
& \simeq Hom_{\mathbf{Set}^{K}}(\eta_K^*(\mu_{K})_*(Tk_1)_*(k_0)_* X,Z)\\
\end{align*}
\end{itemize}
\end{definition}

\begin{proposition}
\label{depset_cat}
 The construction above indeed defines a category, namely, $\mathbf{DepSet_{aff}}$ admits identity morphisms and composition is associative.
\end{proposition}

\begin{proof}

 Consider an object $(I,X)$ of $\mathbf{DepSet_{aff}}$. Since $\eta_I$ is a full inclusion, we have a canonical natural isomorphism $\eta_I^* (\eta_I)_* \simeq id_{\mathbf{Set^{I}}}$. In particular, the identity morphism $id_X : X \to X$ yields an arrow $(I,X) \to (I,X)$ in $\mathbf{DepSet_{aff}}$, that can be checked to define the identity arrow on $(I,X)$.
 
 For associativity of composition, given three composable arrows $(I,X) \to (J,Y) \to (K,Z) \to (L,W)$ in $\mathbf{DepSet_{aff}}$, the first way to compose them yields the following arrow (on the right)
\[\begin{tikzcd}[ampersand replacement=\&]
	{(\mu_{K})_*(Tk_1)_*(k_0)_* X} \& {(k_2)_*Z} \& {(\mu_{L})_*(Tk_2)_* (\mu_{K})_*(Tk_1)_*(k_0)_* X} \\
	{(k_1)_* Y} \&\& {(\mu_{L})_*(Tk_2)_* (k_1)_* Y} \\
	{(\eta_K)_*Z} \& {(\eta_L)_*W} \& {(\mu_{L})_*(Tk_2)_* (\eta_K)_*Z \simeq (k_2)_* Z} \\
	\&\& {(\eta_L)_*W}
	\arrow[from=1-1, to=2-1]
	\arrow[from=1-2, to=3-2]
	\arrow[from=1-3, to=2-3]
	\arrow[from=2-1, to=3-1]
	\arrow[from=2-3, to=3-3]
	\arrow[from=3-3, to=4-3]
\end{tikzcd}\]
while the second way yields the following one (again, on the right):
\[\begin{tikzcd}[ampersand replacement=\&]
	{(k_0)_*X} \& {(\mu_{L})_*(Tk_2)_*(k_1)_* Y} \& {(\mu_{L})_*(T\mu_{L})_*(T^2k_2)_*(Tk_1)_* (k_0)_*X} \\
	{(\eta_J)_*Y} \& {(k_2)_* Z} \& {(\mu_{L})_*(T\mu_{L})_*(T^2k_2)_*(Tk_1)_* (\eta_J)_*Y \simeq (\mu_{L})_*(Tk_2)_*(k_1)_* Y} \\
	\& {(\eta_L)_*W} \& {(k_2)_* Z} \\
	\&\& {(\eta_L)_*W}
	\arrow[from=1-1, to=2-1]
	\arrow[from=1-2, to=2-2]
	\arrow[from=1-3, to=2-3]
	\arrow[from=2-2, to=3-2]
	\arrow[from=2-3, to=3-3]
	\arrow[from=3-3, to=4-3]
\end{tikzcd}\]
The two ways are equal because of associativity of composition in the Kleisli category $\Lambda(\mathbf{Dep_{aff}})$ and in $\mathbf{Set}^{L}$.
\end{proof}

\begin{remark}
 The proof of \cref{depset_cat} suggest that we may also consider any derivator instead of the derivator represented by $\mathbf{Set}$ (given by the $2$-functor $I \mapsto \mathbf{Set}^I$).
\end{remark}

\begin{proposition}
 Together with the projection functor $\mathbf{DepSet_{aff}} \to \Lambda(\mathbf{Dep_{aff}})$, the previous construction defines a dependent operad.
\end{proposition}

\begin{proof}
 Given an inert morphism $k : I \to TJ$ in $\Lambda(\mathbf{Dep_{aff}})$, we have a pullback square

\[\begin{tikzcd}[ampersand replacement=\&]
	J \&\& I \\
	\\
	J \&\& TJ
	\arrow["{k^t}", from=1-1, to=1-3]
	\arrow["{id_J}"', from=1-1, to=3-1]
	\arrow["k", from=1-3, to=3-3]
	\arrow["{\eta_J}"', from=3-1, to=3-3]
\end{tikzcd}\]
by definition.
Since $\eta_j : J \to TJ$ is an opfibration, this is an exact square and we have that
$\eta_J^* k_* \simeq k_t^*$. For $X$ an object of $\mathbf{DepSet_{aff}}$ lying over $I$, that is $X \in \mathbf{Set}^I$, we can consider the object $Y := k_t^* X$ which lies over $J$, and together with the identity map $id_{k_t^* X} \in Hom(k_t^*X,k_t^*X)$, this provides a lifting of $k$ with domain $X$ that is clearly cocartesian.

Conditions 2 and 3 for dependent operads also clearly hold, using that the mapping $I \mapsto \mathbf{Set}^{I}$ sends colimits in $\mathbf{Dep_{aff}}$ to limits in $\mathbf{CAT}$.
\end{proof}

\begin{definition}
\label{depset_sig}
We write $\mathbf{DepSet_{aff}}^*$ for the image under $U_\mathbf{Dep_{aff}} : \mathbf{Op} \to \mathbf{Sig}$, that we take the pullback in the diagram below:
\[\begin{tikzcd}[ampersand replacement=\&]
	{\mathbf{DepSet_{aff}}^*} \&\& {\mathbf{DepSet_{aff}}} \\
	\\
	{\Lambda(\mathbf{Dep_{aff}})^\dagger} \&\& {\Lambda(\mathbf{Dep_{aff}})}
	\arrow[from=1-1, to=1-3]
	\arrow[from=1-1, to=3-1]
	\arrow["\ulcorner"{anchor=center, pos=0.125}, draw=none, from=1-1, to=3-3]
	\arrow[from=1-3, to=3-3]
	\arrow[from=3-1, to=3-3]
\end{tikzcd}\]
\end{definition}

\begin{remark}
 It follows that $\mathbf{DepSet_{aff}}^* \to \Lambda(\mathbf{Dep_{aff}})^\dagger$ is an opfibration, and that it corresponds under the Grothendieck construction to the functor $\Lambda(\mathbf{Dep_{aff}})^\dagger \simeq (\mathbf{Dep_{aff}}^+)^{op} \to \mathbf{CAT}$ mapping $I \in  \mathbf{Dep_{aff}}_s$ to $\mathbf{Set}^{I^{op}}$ (and acting by precomposition).
\end{remark}

\begin{definition}
\label{algebra_affine}
 Given a dependent operad $O : O^\otimes \to \Lambda(\mathbf{Dep_{aff}})$, an algebra (in dependent sets over the corresponding signature) over this dependent operad is defined to be a morphism of dependent operads $\alpha : O^\otimes \to \mathbf{DepSet_{aff}}$: 
\[\begin{tikzcd}[ampersand replacement=\&]
	{O^\otimes} \&\&\&\& {\mathbf{DepSet_{aff}}} \\
	\\
	\&\& {\Lambda(\mathbf{Dep_{aff}})}
	\arrow["\alpha", from=1-1, to=1-5]
	\arrow[from=1-1, to=3-3]
	\arrow[from=1-5, to=3-3]
\end{tikzcd}\]

We write $\mathbf{Alg}_{O}$ for the category of algebras over the dependent operad $O : O^\otimes \to \lambda(\mathbf{Dep_{aff}})$, whose objects are such functors and whose morphisms are the natural transformations between them.
\end{definition}

\begin{example}
 An algebra $\alpha$ for the graded monoids operad $T_\mathbf{GMon} : O_\mathbf{GMon} \to \Lambda(\mathbf{Dep_{aff}})$ is given by a set $B$ (the image of $b$ under the $\alpha : O_\mathbf{GMon} \to \mathbf{DepSet_{aff}}$) and a function $E \to B$ (the image of $e$ under the $\alpha$, which is an element of $\mathbf{Set}^{\mathbb{2}}$, such as restriction to the codomain yields $B$).
\end{example}

\begin{definition}
\label{dependent_set_over_signature_affine}
Given a dependent signature $S : O^\otimes \to \Lambda(\mathbf{Dep_{aff}})^\dagger$, we define a dependent set over this signature to be a morphism from $S$ to $\mathbf{DepSet_{aff}} \to \Lambda(\mathbf{\mathbf{Dep_{aff}}})^\dagger$ (in the category of cocartesian fibrations), pictorially:
\[\begin{tikzcd}[ampersand replacement=\&]
	{O^\otimes} \&\& {\mathbf{DepSet_{aff}}^*} \\
	\\
	\& {\Lambda(\mathbf{Dep_{aff}})^\dagger}
	\arrow[from=1-1, to=1-3]
	\arrow[from=1-1, to=3-2]
	\arrow[from=1-3, to=3-2]
\end{tikzcd}\] 

We write $\mathbf{DepSet^S_{aff}}$ for the category of dependent sets over $S$.
\end{definition}

\begin{proposition}
\label{presheaf_topos}
 $\mathbf{DepSet^S_{aff}}$ is equivalent to a presheaf topos.
\end{proposition}

\begin{proof}
 Consider a split fibration equivalent to $S$, and the corresponding functor $S_t : \Lambda(\mathbf{\mathbf{Dep_{aff}}})^\dagger \to \mathbf{Cat}$. Write $D:  \Lambda(\mathbf{\mathbf{Dep_{aff}}})^\dagger \to \mathbf{CAT}$ for the functor corresponding to $\mathbf{DepSet_{aff}} \to \Lambda(\mathbf{\mathbf{Dep_{aff}}})^\dagger$.
 
 Then, $\mathbf{DepSet^S_{aff}}$ is equivalent to the hom-category $$Hom_{\mathbf{CAT}^{\Lambda(\mathbf{\mathbf{Dep_{aff}}})^\dagger}}(S_t,D) \simeq Hom_{\mathbf{CAT}^{\Lambda(\mathbf{\mathbf{Dep_{aff}}})^\dagger}}(*,D^{S_t})$$
 
By the general formula for exponentials in a functor category with values in a cartesian closed category, we see that, for $I \in \Lambda(\mathbf{\mathbf{Dep_{aff}}})^\dagger$, the following equation holds:

$$D^{S_t}(I) \simeq \int_{J \in \Lambda(\mathbf{\mathbf{Dep_{aff}}})^\dagger} \Pi_{Hom(I,J)} D(J)^{S_t(J)} \simeq \int_{J \in \Lambda(\mathbf{\mathbf{Dep_{aff}}})^\dagger} \Pi_{Hom(I,J)} \mathbf{Set}^{J \times S_t(J)}$$

Equivalently, $D^{S_t}(I)$ is the image under the continuous mapping $C \mapsto Set^C$ of the coend

$$C_I := \int^{J \in \Lambda(\mathbf{\mathbf{Dep_{aff}}})^\dagger} Hom(I,J) \times J \times S_t(J)$$ 

Similarly, for $I \to J$ a map in $\Lambda(\mathbf{\mathbf{Dep_{aff}}})^\dagger$, the corresponding functor $D^{S_t}(I) \to D^{S_t}(J)$ is induced by precomposition (with the functor between the corresponding colimits). This yields a diagram $C_{(-)} : \Lambda(\mathbf{\mathbf{Dep_{aff}}})^\dagger \to \mathbf{Cat}$.

Finally, from the equivalence,
$$\mathbf{DepSet^S_{aff}} \simeq  \underline{Hom}_{\mathbf{CAT}}(*,\varprojlim D^{S_t}) \simeq \varprojlim D^{S_t}$$
we deduce that $\mathbf{DepSet^S_{aff}}$ is equivalent to the presheaf topos on the category that is the limit of the diagram $C_{(-)}$.
\end{proof}

\begin{example}
 The (site of the) presheaf topos corresponding to the dependent signature $O_\mathbf{GMon}$ associated to graded monoids can be recovered from the diagram $$C^\mathbf{GMon} : \Lambda(\mathbf{\mathbf{Dep_{aff}}})^\dagger \to \mathbf{Cat}$$ obtained by right Kan extension along $\Lambda(\mathbf{\mathbf{Dep_{aff}}})^\dagger_0 \to \Lambda(\mathbf{\mathbf{Dep_{aff}}})^\dagger$ from a diagram $$C_0^\mathbf{GMon} : \Lambda(\mathbf{\mathbf{Dep_{aff}}})^\dagger_0 \to \mathbf{Cat}$$ takes only two values different from the empty category $\emptyset$, namely:
 \begin{itemize}
  \item $C_0^\mathbf{GMon}(\mathbb{1}) = \int^{J \in \Lambda(\mathbf{\mathbf{Dep_{aff}}})^\dagger} Hom(\mathbb{1},J) \times J \times S_t(J)$. Computing the following coequalizer,

\[\begin{tikzcd}[ampersand replacement=\&]
	\begin{array}{c}  (Hom(\mathbb{1},\mathbb{1}) \times Hom(\mathbb{1},\mathbb{1}) \times  \mathbb{1})   \amalg (Hom(\mathbb{2},\mathbb{2}) \times Hom(\mathbb{1},\mathbb{2}) \times \mathbb{2})\\ \amalg (Hom(\mathbb{1},\mathbb{2}) \times Hom(\mathbb{1},\mathbb{1}) \times \mathbb{2}) \amalg (Hom(\mathbb{2},\mathbb{1}) \times Hom(\mathbb{1},\mathbb{2}) \times \mathbb{1})\\\simeq \mathbb{1} \amalg \mathbb{2} \end{array} \\
	\\
	{\mathbb{1} \amalg \mathbb{2}} \\
	{\mathbb{1} \amalg \mathbb{2}}
	\arrow["id", shift left, from=1-1, to=3-1]
	\arrow["id"', shift right, from=1-1, to=3-1]
	\arrow[from=3-1, to=4-1]
\end{tikzcd}\]
this reduces to $\mathbb{1} \amalg \mathbb{2}$.
\item $C_0^\mathbf{GMon}(\mathbb{2}) = \int^{J \in \Lambda(\mathbf{\mathbf{Dep_{aff}}})^\dagger} Hom(\mathbb{2},J) \times J \times S_t(J)$. Computing the following coequalizer,
\[\begin{tikzcd}[ampersand replacement=\&]
	\begin{array}{c}  (Hom(\mathbb{1},\mathbb{1}) \times Hom(\mathbb{2},\mathbb{1}) \times  \mathbb{1})   \amalg (Hom(\mathbb{2},\mathbb{2}) \times Hom(\mathbb{2},\mathbb{2}) \times \mathbb{2})\\ \amalg (Hom(\mathbb{1},\mathbb{2}) \times Hom(\mathbb{2},\mathbb{1}) \times \mathbb{2}) \amalg (Hom(\mathbb{2},\mathbb{1}) \times Hom(\mathbb{2},\mathbb{2}) \times \mathbb{1})\\\simeq \mathbb{1} \amalg \mathbb{2} \amalg \mathbb{1} \end{array} \\
	\\
	{\mathbb{1} \amalg \mathbb{2}} \\
	{\mathbb{2}}
	\arrow["{id \amalg id_{\mathbb{1}}}", shift left, from=1-1, to=3-1]
	\arrow["{id \amalg t}"', shift right, from=1-1, to=3-1]
	\arrow[from=3-1, to=4-1]
\end{tikzcd}\]
this reduces to $\mathbb{2}$.
 \end{itemize}

 Therefore, the site is the limit of the diagram given by the inclusion: $$\mathbb{2} \subset \mathbb{1} \amalg \mathbb{2}$$ 
 
 That is, it is the walking arrow category $\mathbb{2}$.
 In other words, the category of dependent sets over the corresponding signature is the presheaf category $\mathbf{Set}^{\mathbb{2}}$ of "graded sets".
\end{example}

Note that, for every dependent operad $O$, the functor $U_{\mathbf{Dep_{aff}}} : \mathbf{Op} \to \mathbf{Sig}$ induces a functor $V^O : \mathbf{Alg}_{O} \to \mathbf{DepSet^S_{aff}}$, where $S := U_{\mathbf{Dep_{aff}}} O$. In our framework, giving that the forgetful functor $V^O$ acts by "precomposition", free algebras are constructed, in a sense, by left Kan extension:

\begin{definition}[{\cite[Definition 2.4.1]{harpaz2015grothendieck}}]
\label{relative_cocomplete}
 Let I be a small category and consider a commutative diagram of categories of the form

\[\begin{tikzcd}[ampersand replacement=\&]
	J \&\& {\mathcal{D}} \\
	\\
	{J^\triangleright} \&\& {\mathcal{C}}
	\arrow["\delta", from=1-1, to=1-3]
	\arrow[from=1-1, to=3-1]
	\arrow["\pi", from=1-3, to=3-3]
	\arrow["\epsilon"', from=3-1, to=3-3]
\end{tikzcd}\]
A \textbf{colimit of $\delta$ relative to $\epsilon$ and $\pi$} is a dashed lift
\[\begin{tikzcd}[ampersand replacement=\&]
	J \&\& {\mathcal{D}} \\
	\\
	{J^\triangleright} \&\& {\mathcal{C}}
	\arrow["\delta", from=1-1, to=1-3]
	\arrow[from=1-1, to=3-1]
	\arrow["\pi", from=1-3, to=3-3]
	\arrow[dashed, from=3-1, to=1-3]
	\arrow["\epsilon"', from=3-1, to=3-3]
\end{tikzcd}\]
which is initial in the category of all such lifts. We will say that $\pi : \mathcal{D} \to \mathcal{C}$ is cocomplete if for every square such as above, $\delta$ admits a colimit relative to $\epsilon$ and $\pi$. 
\end{definition}

\begin{proposition}
\label{free_algebras}
 $V^O$ admits a left adjoint $F^O :\mathbf{DepSet^S_{aff}} \to \mathbf{Alg}_{O}$, where $S := U_{\mathbf{Dep_{aff}}} O$.
\end{proposition}

\begin{proof}
Consider a dependent set $C : S^\otimes \to \mathbf{DepSet_{aff}}$.
 The opfibration  $\mathbf{DepSet_{aff}} \to \Lambda(\mathbf{\mathbf{Dep_{aff}}})$ is cocomplete in the sense of \cref{relative_cocomplete} by \cite[Proposition 2.4.4]{harpaz2015grothendieck} since the presheaf categories $\mathbf{Set}^{I}$ are all cocomplete, so we will be able to construct the left adjoint in a pointwise fashion (as with pointwise Kan extensions).
 Explicitly, consider an object $x \in O^\otimes$, and consider the comma category as follows:

\[\begin{tikzcd}[ampersand replacement=\&]
	{\iota \downarrow x} \&\& {S^\otimes} \\
	\\
	{*} \&\& {O^\otimes}
	\arrow[from=1-1, to=1-3]
	\arrow[from=1-1, to=3-1]
	\arrow[between={0.3}{0.7}, Rightarrow, from=1-3, to=3-1]
	\arrow["\iota", from=1-3, to=3-3]
	\arrow["x"', from=3-1, to=3-3]
\end{tikzcd}\]

Note that the diagram $\iota \downarrow x \to S^\otimes \to O^\otimes \to \Lambda(\mathbf{\mathbf{Dep_{aff}}})$ extends to a cocone with apex $\pi(x)$

We define a diagram $F^O(C) : O^\otimes \to \mathbf{DepSet_{aff}}$ over $\Lambda(\mathbf{\mathbf{Dep_{aff}}})$ on objects by mapping an $x$ such as above to the colimit associated to the diagram below,
\[\begin{tikzcd}[ampersand replacement=\&]
	{\iota \downarrow x} \& {S^\otimes} \& {\mathbf{DepSet_{aff}}} \\
	\\
	{(\iota \downarrow x)^\triangleright} \&\& {\Lambda(\mathbf{Dep_{aff}})}
	\arrow[from=1-1, to=1-2]
	\arrow[from=1-1, to=3-1]
	\arrow[from=1-2, to=1-3]
	\arrow[from=1-3, to=3-3]
	\arrow["h", dashed, from=3-1, to=1-3]
	\arrow[from=3-1, to=3-3]
\end{tikzcd}\]
i.e. to the apex of the cocone corresponding to the initial lift $h$.

This construction can be promoted to a functor $O^\otimes \to \mathbf{DepSet_{aff}}$ using the universal properties of these initial lifts, and the canonical inclusions $\iota \downarrow x \to \iota \downarrow y$, for $x \to y$ an arrow in $O^\otimes$. This defines the functor $F^O$.

If the dependent set $C : S^\otimes \to \mathbf{DepSet_{aff}}$ is the image of an algebra $A : O^\otimes \to \mathbf{DepSet_{aff}}$ through the forgetful functor $U_{\mathbf{Dep_{aff}}}$, then, there are canonical lifts as in the diagram above, therefore we have maps from the apex of the cocone associated to the initial lifts to the apex of the cocone corresponding to these canonical lifts. Those assemble into a natural transformation from the functor $O^\otimes \to \mathbf{DepSet_{aff}}$ constructed "freely" above from $C$ to the algebra $A$ we started with. Globally, this yields a counit $F^O \circ V^O \to id_{\mathbf{Alg}_{O}}$

Likewise, the object $x \in O^\otimes$ considered above is also an object of $S^\otimes$, and we get a mediating arrow from the apex $l$ of the initial cocone to the image of $x$ under $S^\otimes \to \mathbf{DepSet_{aff}}$ from the universal property of this cocone. This defines the component of a natural transformation $C \to  U_{\mathbf{Dep_{aff}}} \circ F^O C$. Finally, these yield the unit  $id_{\mathbf{DepSet^S_{aff}}} \to  V^O \circ F^O$.

The triangular identities follow similarly to the case of "real" pointwise left Kan extensions, concluding the proof of that $F^O$ is left adjoint to $V^O$.
\end{proof}

\begin{theorem}
\label{monadicity}
 $V^O$ is monadic.
\end{theorem}

\begin{proof}
 We will apply Beck's monadicity theorem, namely, we will check that $V^O$ is conservative and creates $V^O$-split coequalizers. That $V^O$ is conservative follows directly from the facts that the inclusion $S^\otimes \subset O^\otimes$ is identity-on-objects and that natural transformations are invertible precisely when they are pointwise invertible.
 
 To see that $V^O$ creates $V^O$-split coequalizer, we consider a parallel pair of morphisms in $\mathbf{Alg}_O$, which is mapped to a split coequalizer by $V$, as below:

\[\begin{tikzcd}[ampersand replacement=\&]
	A \&\& B \&\& \\
	\\
	V^OA \&\& V^OB \&\& {C_0}
	\arrow["g"', shift right=2, from=1-1, to=1-3]
	\arrow["f", shift left=2, from=1-1, to=1-3]
	\arrow["Vf", shift left=2, from=3-1, to=3-3]
	\arrow["Vg"', shift right=2, from=3-1, to=3-3]
	\arrow[from=3-3, to=3-1]
	\arrow[from=3-3, to=3-5]
	\arrow[shift right, curve={height=6pt}, from=3-5, to=3-3]
\end{tikzcd}\]

Now consider an arrow $x \to y$ in $O^\otimes$ over an morphism $k : I \to TJ$ in $\Lambda(\mathbf{\mathbf{Dep_{aff}}})$. We want to extend $C_0$ to an algebra, by defining the action of $x \to y$, i.e by constructing an element of $$Hom_\mathbf{DepSet_{aff}}(C_0(x),C_0(y)) \simeq Hom_{\mathbf{Set}^{J}}(\eta_J^* k_* C_0(x),C_0(y))$$
Since a split coequalizer is an absolute colimit, we have an equalizer diagram
\[\begin{tikzcd}[ampersand replacement=\&]
	{Hom_\mathbf{DepSet}(C_0(x),C_0(y))} \& {Hom_\mathbf{DepSet}(V^OB(x),C_0(y))} \& {Hom_\mathbf{DepSet}(V^OA(x),C_0(y))}
	\arrow[from=1-1, to=1-2]
	\arrow[shift right=2, from=1-2, to=1-3]
	\arrow[shift left=2, from=1-2, to=1-3]
\end{tikzcd}\]
so that the morphisms $V^O(A)(x) \to V^O(A)(y) \to C_0(y)$ and $V^O(B)(x) \to V^O(B)(y) \to C_0(y)$ stemming from the algebra structures on $A$ and $B$ induce together an arrow $C_0(x) \to C_0(y)$.
It is clear from the universal properties of these equalizers that the arrows so-defined assemble into a functor $O^\otimes \to \mathbf{DepSet_{aff}}$ which extends $C_0 : S^\otimes \to \mathbf{DepSet_{aff}}$ into an algebra $C$ (with underlying dependent set $C_0$). Checking that this defines a coequalizer in the category of algebras can be done by using the universal properties of the equalizers above to extend the factorization of a cocone through $C_0$ in the category $\mathbf{DepSet^S_{aff}}$ to a factorization in $\mathbf{Alg}_O$.
 
 This shows that the assumptions for the monadicity theorem are satisfied, and concludes the proof.
\end{proof}

As a corollary, we record the following result:

\begin{corollary}
\label{algebras_loc_presentable}
 $\mathbf{Alg}_{O}$ is locally presentable.
\end{corollary}

\section{Dependent operads: the general case}

In this section, we walk away from the usual notion of operads in order to allow variable duplication. In particular, the theory of operads in \cite{barwick2018operator} and the notion of inert/active orthogonal factorization system in a moment category, as introduced by Berger \cite{berger2021moment}, cannot be applied directly. Nevertheless, the theory can be adapted in the specific case of finite posets, but we do not pursue a general notion of dependent operator category in this setting.

We define $\mathbf{Dep} := \mathbf{FinPos}$.
We consider the assignment $$P \mapsto M(P) := (\mathbb{2}^{P^\triangleleft})^{op}$$ for $P$ a finite poset, where we think of $\mathbb{2}$ as the full subcategory of $\mathbf{Set}$ spanned by the empty set and a one-element set. We extend it to a covariant functor $\mathbf{FinPos} \to \mathbf{FinPos}$ by defining the action on an arrow $P \to Q$ in $\mathbf{FinPos}$ as in the diagram below
\[\begin{tikzcd}[ampersand replacement=\&]
	{M(P) =  (\mathbb{2}^{P^\triangleleft})^{op}} \&\& {M(Q) =  (\mathbb{2}^{Q^\triangleleft})^{op}} \\
	\\
	{(\mathbf{Set}^{P^\triangleleft})^{op}} \&\& {(\mathbf{Set}^{Q^\triangleleft})^{op}}
	\arrow["{M(f)}", dashed, from=1-1, to=1-3]
	\arrow[hook, from=1-1, to=3-1]
	\arrow["{\mathbf{Lan}^{op}}"', from=3-1, to=3-3]
	\arrow["{(?^{Q^\triangleleft})^{op}}"', from=3-3, to=1-3]
\end{tikzcd}\]
where $? : \mathbf{Set} \to \mathbb{2}$ is the functor mapping the empty set to $0$ and the other sets to $1$.

\begin{remark}
 Intuitively, variable duplication is modeled by the fact that, for a map $I \to M(J)$ in $\mathbf{Dep}$, a single object of $I$ can be used to "define" multiple objects of $J$ (namely a "product" of corepresentables). By functoriality, if $i \in I$ "defines" some $j \in I$, it automatically defines its dependencies $j \to j'$. In particular, it is not necessary to add explicitly a condition on the map $I \to M(J)$ to ensure that dependencies are carried over, this is why we work with all maps between finite posets. The initial object of $M(J)$, corresponding to the corepresentable functor on the cone point of $J^\triangleleft$ is used to "throw away" an object $i \in I$. As expected, throwing away $i \in I$ implies throwing away all objects that depend on it.
\end{remark}

\begin{proposition}
 $M$ is canonically endowed with the structure of a monad.
\end{proposition}

\begin{proof}
 The component of the unit at a finite poset $P$ is given by the composite $$P \to P^\triangleleft \to (\mathbb{2}^{P^\triangleleft})^{op}$$ where the second map is the $\mathbb{2}$-enriched covariant Yoneda embedding. Given a functor between small categories $F: C \to D$, the following square commutes "only" up to canonical natural isomorphism:

\[\begin{tikzcd}[ampersand replacement=\&]
	C \&\& D \\
	\\
	{\mathbf{Set}^{C^{op}}} \&\& {\mathbf{Set}^{D^{op}}}
	\arrow["F", from=1-1, to=1-3]
	\arrow[from=1-1, to=3-1]
	\arrow["\simeq"{description}, draw=none, from=1-1, to=3-3]
	\arrow[from=1-3, to=3-3]
	\arrow["{\mathbf{Lan}_F}"', from=3-1, to=3-3]
\end{tikzcd}\]
 But, in our case, since there are no non-trivial isomorphisms in $\mathbb{2}$, we get strictly commutative squares
\[\begin{tikzcd}[ampersand replacement=\&]
	P \&\& Q \\
	\\
	{M(P) = (\mathbb{2}^{P^\triangleleft})^{op}} \&\& {M(Q) = (\mathbb{2}^{Q^\triangleleft})^{op}}
	\arrow[from=1-1, to=1-3]
	\arrow[from=1-1, to=3-1]
	\arrow[from=1-3, to=3-3]
	\arrow["{M(f)}", from=3-1, to=3-3]
\end{tikzcd}\]
 which proves naturality.
 
 The multiplication is defined by Yoneda extension (i.e., by right Kan extension along the covariant Yoneda embedding):
\[\begin{tikzcd}[ampersand replacement=\&]
	\& {(\mathbb{2}^{P^\triangleleft})^{op}} \&\& {(\mathbf{Set}^{P^\triangleleft})^{op}} \& {(\mathbb{2}^{P^\triangleleft})^{op}} \\
	\\
	{(\mathbb{2}^{(\mathbb{2}^{P^\triangleleft})^{op}})^{op}} \& {(\mathbf{Set}^{(\mathbb{2}^{P^\triangleleft})^{op}})^{op}}
	\arrow[from=1-2, to=1-4]
	\arrow["{\mathbf{y}}"', from=1-2, to=3-2]
	\arrow["{(?^{P^\triangleleft})^{op}}"', from=1-4, to=1-5]
	\arrow[from=3-1, to=3-2]
	\arrow[dashed, from=3-2, to=1-4]
\end{tikzcd}\]
 
 Again, the following diagram involving right Kan extensions
\[\begin{tikzcd}[ampersand replacement=\&]
	{M(P)} \&\& {(\mathbf{Set}^{P^\triangleleft})^{op}} \&\& \\
	\&\& {M(Q)} \&\& {(\mathbf{Set}^{Q^\triangleleft})^{op}} \\
	{(\mathbf{Set}^{M(P)})^{op}} \\
	\&\& {(\mathbf{Set}^{M(Q)})^{op}}
	\arrow[from=1-1, to=1-3]
	\arrow["{M(f)}"{description}, from=1-1, to=2-3]
	\arrow["{\mathbf{y}}"', from=1-1, to=3-1]
	\arrow[from=1-3, to=2-5]
	\arrow[from=2-3, to=2-5]
	\arrow["{\mathbf{y}}", from=2-3, to=4-3]
	\arrow[dashed, from=3-1, to=1-3]
	\arrow["{(M(f)_!)^{op}}"', from=3-1, to=4-3]
	\arrow[dashed, from=4-3, to=2-5]
\end{tikzcd}\]
only commutes up to canonical natural isomorphisms, but this is enough to conclude for "strict" naturality in our case.

The unit law for the monad follows from the fact that the universal $2$-cell in a Kan extension along the Yoneda embedding is always invertible (as for any Kan extension along a fully faithful functor). Associativity of the multiplication boils down to pseudo-functoriality of the right Kan extension mapping $$F \mapsto (F_!)^{op}  : (\mathbf{Set}^{C})^{op} \to (\mathbf{Set}^{D})^{op}$$ and more specifically to the corresponding associativity up to natural isomorphisms. 
\end{proof}

We observe that the monad structure on $M$ restricts on the endofunctor mapping $P$ to $M(P) \setminus {\emptyset}$ (i.e. the full subcategory of $M(P)$ without the terminal object). Since this terminal object is meaningless for the intended use of $M$, we make the choice of discarding it, and we will still write $M$ for this second monad.

\begin{remark}
 We have a dense inclusion $\mathbf{Dep_{aff}^{elm}} \subset \mathbf{Dep}$. However, $M$ is not a monad with arities (as studied in \cite{berger2012monads}) relative to this inclusion. If we consider instead $\mathbf{Dep_{aff}^{elm}}'$ obtained from $\mathbf{Dep_{aff}^{elm}}$ by adding the empty poset as an initial object, the inclusion $\mathbf{Dep_{aff}^{elm}}' \subset \mathbf{Dep}$ is still dense and $M$ is a monad with arities with respect to it: it is not difficult to check that the associated nerve functor preserves the empty poset, coproduct of posets (since the posets in $\mathbf{Dep_{aff}^{elm}}$ have initial objects; in particular they are connected) and pushout of two full inclusions (which are easy to describe), and we believe that $M$ preserve any connected colimit.
 
 Note that, for other reasons, we have made the choice not to include the empty poset in the elementary objects of $\mathbf{Dep_aff}$. The monad $M$ is actually the composite of two monads through a distributive law: the monad adding an initial object, and the variable duplication monad defined as $M$ but without the $(-)^\triangleleft$ part. The second is already a monad with arities with respect to $\mathbf{Dep_{aff}^{elm}} \subset \mathbf{Dep}$.
\end{remark}

Write $\Lambda(\mathbf{Dep})$ for the Kleisli category associated to the monad $M$.

\begin{remark}
 We will think of objects $X$ of $M(P)$, for $P \in \mathbf{FinPos}$, as "products" of elements of $P$: for every element $p$ of $P$, either it is a "component" of $X$ or it is not. The unit of the monad sends a $p$ to the element $X_p \in M(P)$ whose unique "component" is $p$. The multiplication takes a $\chi \in MM(P)$ to $X_\chi \in M(P)$ whose components are $p$ are those that appear at least once in a component of a $\chi$.
\end{remark}

\begin{definition}
 We will say that a map $f : P \to M(Q)$ from $P$ to $Q$ is
  \begin{itemize}
   \item active if it is a free map, i.e. of the form $P \to Q \to M(Q)$ where the second arrow is the unit $\eta_Q : Q \to M(Q)$.
   \item inert if the projection $p_0$ in the diagram

\[\begin{tikzcd}[ampersand replacement=\&]
	{\eta_Q \downarrow f} \&\& Q \\
	\\
	P \&\& {M(Q)}
	\arrow["{p_0}", from=1-1, to=1-3]
	\arrow["{p_1}"', from=1-1, to=3-1]
	\arrow[between={0.3}{0.7}, Rightarrow, from=1-3, to=3-1]
	\arrow["{\eta_Q}", from=1-3, to=3-3]
	\arrow["f"', from=3-1, to=3-3]
\end{tikzcd}\]
   has a left adjoint right inverse.
\end{itemize}
\end{definition}

\begin{remark}
\label{lari}
 If $f : P \to M(Q)$ is inert, then, for every $q \in Q$, we can form the pullback on top in the diagram below,

\[\begin{tikzcd}[ampersand replacement=\&]
	{\eta_Q \circ q \downarrow f} \&\& {*} \\
	\\
	{\eta_Q \downarrow f} \&\& Q \\
	\\
	P \&\& {M(Q)}
	\arrow[from=1-1, to=1-3]
	\arrow[from=1-1, to=3-1]
	\arrow["\ulcorner"{anchor=center, pos=0.125}, draw=none, from=1-1, to=3-3]
	\arrow["q", from=1-3, to=3-3]
	\arrow["{p_0}", from=3-1, to=3-3]
	\arrow["{p_1}"', from=3-1, to=5-1]
	\arrow[between={0.3}{0.7}, Rightarrow, nfold, from=3-3, to=5-1]
	\arrow["{\eta_Q}", from=3-3, to=5-3]
	\arrow["f"', from=5-1, to=5-3]
\end{tikzcd}\]
and we get a left adjoint right inverse to the projection $\eta_Q \circ q \downarrow f \to *$, so that the domain has an initial object.

Conversely, if every such pullback admits an initial object, the corresponding adjunctions clearly assemble in a left adjoint right inverse for $p_0$.
\end{remark}

\begin{remark}
 It is unreasonable to expect these two classes of maps to form an orthogonal factorization system, as required by the definition of a moment category \cite{berger2021moment}. This is because we may duplicate an object an arbitrary number of times before using it. However, we still expect to get a weak factorization system as we establish below.
\end{remark}

\begin{lemma}
\label{inert_active_fac}
 Every map in $\Lambda(\mathbf{Dep})$ factors as an inert map followed by an active one.
\end{lemma}

\begin{proof}
 Consider a map $f : P \to M(Q)$, and form the following comma square
\[\begin{tikzcd}[ampersand replacement=\&]
	{R := \eta_Q \downarrow f} \&\& Q \\
	\\
	P \&\& {M(Q)}
	\arrow["{p_0}", from=1-1, to=1-3]
	\arrow["{p_1}"', from=1-1, to=3-1]
	\arrow[between={0.3}{0.7}, Rightarrow, from=1-3, to=3-1]
	\arrow["{\eta_Q}", from=1-3, to=3-3]
	\arrow["f"', from=3-1, to=3-3]
\end{tikzcd}\]
as in the definition above.

This provides an active map $h : R \to Q$.
We define a map $g : P \to M(R)$ as the transpose of the functor 
$$P \times (R^{op}) \to \mathbb{2}$$
that maps $(p,r) \in P \times R^{op}$ to $1$ if $p_1(r) \leq p$ and to $0$ otherwise, and that maps pairs of the form $(p,*)$ to $0$. This is clearly covariantly functorial in $P$ and contravariantly so in $R$. The composite $P \to M(R) \to M(Q)$ takes $p \in P$ to the  object of $M(Q)$ given by the "union" of the representables $\mathbf{y}(q)$ such $\mathbf{y}(q) \leq f(p)$, which is precisely $f(p)$, so that we have indeed defined a factorization of $f$.

Now, forming the comma square
\[\begin{tikzcd}[ampersand replacement=\&]
	{\eta_R \downarrow g} \&\& R \\
	\\
	P \&\& {M(R)}
	\arrow["{q_0}", from=1-1, to=1-3]
	\arrow["{q_1}"', from=1-1, to=3-1]
	\arrow[between={0.3}{0.7}, Rightarrow, from=1-3, to=3-1]
	\arrow["{\eta_R}", from=1-3, to=3-3]
	\arrow["g"', from=3-1, to=3-3]
\end{tikzcd}\]
we may define a functor $R \to \eta_R \downarrow g$ mapping $r \in R$ to the arrow $i_r : \mathbf{y}(r) \to g(p)$ in $M(Q)$ where $p := p_1(r)$. We then observe that an arrow $\alpha : \mathbf{y}(r') \to g(p')$ yields an object of $\eta_R \downarrow g$ such that $i_r \leq \alpha$ if and only if $r \leq r'$ in $R$. Indeed, the direct implication is evident, and, for the converse, if $r \leq r'$, then,  $p := p_1(r) \leq p_1(r')$, and, by definition of $g$, $p_1(r') \leq p'$ as the subposet $P_{\leq p_1(r')}$ of $P$ formed by the objects $q \leq p_1(r')$ is contained in $P_{\leq p'}$. This proves that $g$ is inert, concluding the proof.
\end{proof}

\begin{lemma}
\label{inert_llp}
 The inert maps are the maps having the left lifting property against the active maps. 
\end{lemma}

\begin{proof}
 Consider a lifting problem in the Kleisli category of $M$, represented as follows,

\[\begin{tikzcd}[ampersand replacement=\&]
	P \&\& {M(R)} \\
	\\
	{M(Q)} \&\& {M^2(S)}
	\arrow[from=1-1, to=1-3]
	\arrow["f"', from=1-1, to=3-1]
	\arrow["{M(g)}", from=1-3, to=3-3]
	\arrow[from=3-1, to=3-3]
\end{tikzcd}\]
where $f$ is inert and $g$ is active.

Now consider an object $q \in Q$. By assumption on $f$, it is a "duplicate" of a unique object $p$ in $P$. Given an object $s \in S$ such that $s$ is a "component" of the image of $q$ under the bottom arrow, and since the diagram commutes, $p$ must be mapped by the top arrow to some object of $M(R)$ having some component $r$ mapped to $s$ by $g : R \to S$. We define a mapping $Q \to M(R)$ by sending every such $q$ to the "product" of those $r$ having this property. We claim that this provides a solution $h : Q \to M(R)$ to the lifting problem.

Indeed, the top triangle commutes because, given an object $p \in P$, then for every $r \in R$ that is a component of the image of $p$ under the top arrow, the subset of objects $q \in Q$ in $f(p)$ whose image under the bottom arrow contains $g(r) \in S$ must be non-empty (from the commutation of the square).

For commutation of the bottom triangle, we crucially rely on the fact that $g$ is a free map (we already wrote $g : R \to S$ for the underlying map) to see that the only $s \in S$ that are component of the image $I_q$ of a given $q \in Q$ under the bottom map are those in the image under $g$ of $h(q)$ (namely, if $s_0 \in S$ is in the image of the composite $g \circ h$, then, by definition of $h$, it must come from an $r_0$ whose image under $g$ is in $I_q$, so $s_0$ is in $I_q$).

For the converse, we consider a map $f : P \to M(Q)$ that is not inert. Making use of \cref{lari}, we have two cases:
\begin{itemize}
 \item There exists a $q \in Q$ such that $\eta_Q \circ q \downarrow f$ is empty.
 \item There exists a $q \in Q$ such that $\eta_Q \circ q \downarrow f$ is not empty and does not have an initial object.
\end{itemize}

In the first case, we may consider the lifting problem below,
\[\begin{tikzcd}[ampersand replacement=\&]
	P \&\& {M(Q \setminus Q_{\geq q})} \\
	\&\& {M(Q)} \\
	{M(Q)} \&\& {MM(Q)}
	\arrow["{f'}", from=1-1, to=1-3]
	\arrow["f"', from=1-1, to=3-1]
	\arrow["{M(\iota_q)}", from=1-3, to=2-3]
	\arrow["{\eta_Q}", from=2-3, to=3-3]
	\arrow["{\eta_Q}"', from=3-1, to=3-3]
\end{tikzcd}\]
where $Q_{\geq q}$ is the subposet of $Q$ whose elements are greater than $q$ (with $\iota_q : Q \setminus Q_{\geq q} \to Q$ the inclusion), and where the map $f' : P \to M(Q \setminus Q_{\leq q}$ is induced from $f$ since $\eta_Q \circ q \downarrow f$ is empty.

Suppose there is a solution $h : Q \to M(Q \setminus Q_{\geq q})$. In particular, it makes the bottom triangle commute, but that is clearly impossible, as $h$ has to map $q \in Q$ to the "product" of some $q' \neq q$. The previous lifting problem against the active map $\iota_q$ has no solution, therefore, $f$ is not inert.

In the second case, since $\eta_Q \circ q \downarrow f$ is finite and non-empty by assumption, it must have at least two distinct minimal elements $m_0$ and $m_1$.
We define a map $P \to \mathbf{Csp}$, where $\mathbf{Csp}$ is the cospan poset pictured below, by mapping $p \in P$ to $0$ if $p \geq m_0$, to $1$ if $p \geq m_1$ and to $*$ otherwise.
\[\begin{tikzcd}[ampersand replacement=\&]
	0 \&\& 1 \\
	\& {*}
	\arrow[from=2-2, to=1-1]
	\arrow[from=2-2, to=1-3]
\end{tikzcd}\]

Similarly, we define a map $Q \to \mathbb{2}$ into the walking arrow category by mapping $q' \in Q$ to the codomain if $q' \geq q$ and to the domain otherwise.
Now, writing $c: \mathbf{Csp} \to \mathbb{2}$ for the functor mapping $*$ to the source, and mapping $0$ and $1$ to the target, we have a commutative square:

\[\begin{tikzcd}[ampersand replacement=\&]
	P \& {\mathbf{Csp}} \& {M(\mathbf{Csp})} \\
	\&\& {M(\mathbb{2})} \\
	{M(Q)} \& {M(\mathbb{2})} \& {MM(\mathbb{2})}
	\arrow[from=1-1, to=1-2]
	\arrow["f"', from=1-1, to=3-1]
	\arrow["{\eta_{\mathbf{Csp}}}", from=1-2, to=1-3]
	\arrow["{M(c)}", from=1-3, to=2-3]
	\arrow["{\eta_{M(\mathbb{2})}}", from=2-3, to=3-3]
	\arrow[from=3-1, to=3-2]
	\arrow["{\eta_{M(\mathbb{2})}}"', from=3-2, to=3-3]
\end{tikzcd}\]

It is clear by construction that there is no solution to this lifting problem. Hence, $f$ is not inert in this second case either. 
\end{proof}

\begin{lemma}
\label{retract_active}
 The active maps in $\Lambda(\mathbf{Dep})$ are closed under retract. 
\end{lemma}

\begin{proof}
 Consider a retract $I' \to J'$ of an active map $I \to J$ in $\Lambda(\mathbf{Dep})$ as below,

\[\begin{tikzcd}[ampersand replacement=\&]
	{I'} \&\& I \&\& {I'} \\
	{K'} \&\&\&\& {K'} \\
	{J'} \&\& J \&\& {J'}
	\arrow[from=1-1, to=1-3]
	\arrow[from=1-1, to=2-1]
	\arrow[from=1-3, to=1-5]
	\arrow[from=1-3, to=3-3]
	\arrow[from=1-5, to=2-5]
	\arrow[dashed, from=2-1, to=1-3]
	\arrow[from=2-1, to=3-1]
	\arrow[from=2-5, to=3-5]
	\arrow[from=3-1, to=3-3]
	\arrow[from=3-3, to=3-5]
\end{tikzcd}\]
 where we have factored $I' \to J'$ as an inert map $I' \to K'$ followed by an active map $K' \to J'$, and where we have considered the indicated dotted lift.
 
 We observe that the map $K' \to I \to I' \to K'$ is idempotent. Now consider an object $k \in K'$. By definition it comes by duplication from an object $i \in I'$, and $i$ is mapped by the inert map $I' \to TK'$ to some "product" $k_0 \times ... \times k_n$ where $k$ appears. The action of $TK' \to TTK'\to TK'$ therefore sends this "product" to itself. We want to show that it actually comes from a permutation of the set $\{k_0,...,k_n\}$. Suppose it is not the case: this means that some $k_i$ are discarded and some others are sent to "products" of size at least $2$ (but strictly less than $n+1$). If $k_{i_0}$ is mapped to such a "product" $k_{j_0} \times ... \times k_{j_m}$, then, again, we would like to show that the action of $TK' \to TK'$ on this "product" is by permutation of the set $\{k_{j_0},...,k_{j_m}\}$. Since the iterative process we just sketched strictly decreases the size of the considered set, it must reach an end, namely a step where it is indeed the case that the action comes from a permutation of the set of "components". Ultimately, this proves that each $k_i$ is the image (as a "product" of one object) of some $k_j$ under $K' \to TK'$. This means that  $K' \to TK'$ is active and comes from an isomorphism.
 
 It follows directly from the invertibility of $I' \to K' \to I'$ and $K' \to I' \to K'$ that $I' \to K'$ is an isomorphism. In particular, $I' \to J'$ is active. 
\end{proof}

\begin{proposition}
 The inert maps and the active maps are respectively the left and right class of a weak factorization system on $\Lambda(\mathbf{Dep})$.
\end{proposition}

\begin{proof}
 We have already seen in \cref{inert_active_fac} that every map in $\Lambda(\mathbf{Dep})$ factors as an inert map followed by an active one, and that the inert maps are characterized by having the left lifting property against the active ones (see \cref{inert_llp}).
 
 Therefore, to be able to conclude, we must show that the active maps are also characterized by having the right lifting property against the inert ones. Again, we make use of the usual retract argument, that is, given a map $f : I \to J$ having the right lifting property against the inert morphisms, we factor it as an inert morphism $g : I \to K$ followed by an active one $h : K \to J$. We then consider a solution to the following lifting problem, 
\[\begin{tikzcd}[ampersand replacement=\&]
	I \&\& I \\
	\\
	K \&\& J
	\arrow["id", from=1-1, to=1-3]
	\arrow["g"', from=1-1, to=3-1]
	\arrow["f", from=1-3, to=3-3]
	\arrow[dashed, from=3-1, to=1-3]
	\arrow["h"', from=3-1, to=3-3]
\end{tikzcd}\]
 showing that $f$ is a retract of $h$. We may therefore conclude by \cref{retract_active}. 
\end{proof}

\begin{corollary}
 The class of inert maps is closed under composition.
\end{corollary}
 Given an inert morphism $k : I \to MJ$ in $\Lambda(\mathbf{Dep})$, we have a comma square

\[\begin{tikzcd}[ampersand replacement=\&]
	{\eta_J \downarrow k} \&\& J \\
	\\
	I \&\& {M(J)}
	\arrow["{p_0}", from=1-1, to=1-3]
	\arrow["{p_1}"', from=1-1, to=3-1]
	\arrow[between={0.3}{0.7}, Rightarrow, nfold, from=1-3, to=3-1]
	\arrow["{\eta_J}", from=1-3, to=3-3]
	\arrow["k"', from=3-1, to=3-3]
\end{tikzcd}\]
where $p_0$ admits a left adjoint right inverse $L$ by definition.

In particular, the top square in the diagram below is exact:
\[\begin{tikzcd}[ampersand replacement=\&]
	J \&\& J \\
	\\
	{\eta_J \downarrow k} \&\& J \\
	\\
	I \&\& {M(J)}
	\arrow["id", from=1-1, to=1-3]
	\arrow["L"', from=1-1, to=3-1]
	\arrow["id", from=1-3, to=3-3]
	\arrow["{p_0}", from=3-1, to=3-3]
	\arrow["{p_1}"', from=3-1, to=5-1]
	\arrow[between={0.3}{0.7}, Rightarrow, nfold, from=3-3, to=5-1]
	\arrow["{\eta_J}", from=3-3, to=5-3]
	\arrow["k"', from=5-1, to=5-3]
\end{tikzcd}\]

Therefore, the composite is exact. We will write $k_t$ for the composite $p_1 \circ L$

As in the affine case, we write $\Lambda(\mathbf{Dep})^\dagger$ for the wide subcategory of $\Lambda(\mathbf{Dep})$ whose maps are the inert maps.

\begin{proposition}
 The mapping from $(\Lambda(\mathbf{Dep})^\dagger)^{op}$ to $\mathbf{Dep}$ which is the identity on object and maps $k$ to $k_t$ defines a faithful functor.
\end{proposition}

\begin{proof}
 For functoriality, we take advantage of the categorical structure of the Kleisli category $\Lambda(\mathbf{Dep})$ and of the previous observation regarding the exact square to see that, for $k : I \to M(J)$ and $k' : J \to M(K)$ two composable arrows in $\Lambda(\mathbf{Dep})$, we have the following natural isomorphisms:
 $$(k' \circ k)_t^* \simeq  \eta_K^* \circ (k'\circ k)_*  \simeq \eta_K^* \circ k'_* \circ \eta_J^* \circ k_* \simeq (k'_t)^* \circ k_t^*$$
 
 In particular, we also have $((k' \circ k)_t)_! \simeq (k_t)_! \circ (k'_t)_!$, so we obtain  $k' \circ k)_t \simeq k_t \circ k'_t$ by reasoning on representables (since left Kan extension maps representables to representables). This canonical natural isomorphism between arrows in $\mathbf{FinPos}$ must be the identity since there are no non-trivial isomorphisms in a poset, proving that $k' \circ k)_t = k_t \circ k'_t$. The case for the identity is similar, so we can conclude that we have indeed defined a functor.
 
 For faithfulness, we observe that, given $k_0,k_1 : I \to M(J)$ such that $(k_0)_t = (k_1)_t$, $(\eta_J)_* \circ \eta_J^* \simeq id_{\mathbf{Set}^J}$ because $\eta_J : J \to M(J)$ is a fully faithful functor. In particular, $(k_0)_t = (k_1)_t$ implies the following natural isomorphisms:
 $$\eta_J^*  \circ (k_0)_* \simeq (k_0)_t^* = (k_1)_t^* \simeq \eta_J^*  \circ (k_1)_*$$
 Hence, $(k_0)_* \simeq (k_1)_*$ and we conclude just as before that $k_0 = k_1$.
\end{proof}

We will write $\mathbf{Dep}^+$ for the image of the functor $(-)_t : (\Lambda(\mathbf{Dep})^\dagger)^{op} \to \mathbf{Dep}$, which is therefore isomorphic to $(\Lambda(\mathbf{Dep})^\dagger)^{op}$ via this transposition functor.

\section{Dependent operads over $\mathbf{Dep}$}

\begin{definition}
 A dependent $(\infty-)$operad is an inner fibration $O^\otimes \to \mathbf{N}\Lambda(\mathbf{Dep})$ satisfying the following conditions:
 \begin{itemize}
  \item For every inert morphism $\phi : I \to J$ in $\Lambda(\mathbf{Dep})$, and every object $x$ of $O^\otimes$ lying over $I$, there exists a cocartesian lift $x \to x'$ of $\phi$.
  \item For any objects $I$ and $J$ of $\Lambda(\mathbf{Dep})$ and for any map $\phi : I \to J$, any objects $x \in O^\otimes_I$ and $y \in O^\otimes_J$ and any family of cocartesian edges $\{y \to y_i \;|\; i \in \mathbf{ElD}J\}$ lying over the morphisms in $\mathbf{ElD}J$ (transposed as inert morphisms of $\Lambda(\mathbf{Dep})$), the induced map
  $$Hom^\phi(x,y) \to \varprojlim_{\rho_i \in \mathbf{ElD}J} Hom(x,y_i)^{\rho_i \circ \phi}$$
  is an equivalence.
  \item For any object $I \in \Lambda(\mathbf{Dep})$ and the family of cocartesian edges $\{x \to x_k | x \in O^\otimes_I, k \in \mathbf{ElD}I\}$ lying over the morphisms $\rho_k : I \to I_k$, obtained from transposing monomorphisms $I_k \to I$ in $\mathbf{ElD}I$, is such that the induced map
  $$O^\otimes_I \to \varprojlim_{\rho_k \in \mathbf{ElD}I} O^\otimes_{I_k}$$
  is an equivalence.
 \end{itemize}
 
 We write $\mathbf{Op}$ for the category of dependent operads, where the morphisms are the functors $O_0^\otimes \to O_1^\otimes$ above $\mathbf{N}\Lambda(\mathbf{Dep})$ that preserve cocartesian morphisms lying above an inert morphism.
\end{definition}

\begin{definition}
 A dependent signature over $\Phi$ is a cocartesian fibration $S^\otimes \to \mathbf{N}\Lambda(\mathbf{Dep})^\dagger$ satisfying the following conditions:
 \begin{itemize}
  \item For any objects $I$ and $J$ of $\Lambda(\mathbf{Dep})^\dagger$ and any map $\phi : I \to J$, any objects $x \in S^\otimes_I$ and $y \in S^\otimes_J$ and any family of cocartesian edges $\{y \to y_i \;|\; i \in \mathbf{ElD}J\}$ lying over the morphisms in $\mathbf{ElD}J$ (transposed as inert morphisms of $\Lambda(\mathbf{Dep})$), the induced map
  $$Hom^\phi(x,y) \to \varprojlim_{\rho_i \in \mathbf{ElD}J} Hom(x,y_i)^{\rho_i \circ \phi}$$
  is an equivalence.
  \item For any object $I \in \Lambda(\mathbf{Dep})^\dagger$, the family of cocartesian edges $\{x \to x_k | x \in S^\otimes_I, k \in \mathbf{ElD}I\}$ lying over the morphisms $\rho_k : I \to I_k$, obtained from transposing monomorphisms $I_k \to I$ in $\mathbf{ElD}I$, is such that the induced map
  $$S^\otimes_I \to \varprojlim_{\rho_k \in \mathbf{ElD}I} S^\otimes_{I_k}$$
  is an equivalence.
 \end{itemize}
 
 We write $\mathbf{Sig}$ for the category of dependent signatures, where the morphisms are the functors $S_0^\otimes \to S_1^\otimes$ above $\mathbf{N}\Lambda(\mathbf{Dep})^\dagger$ that preserve cocartesian morphisms.
\end{definition}

There is a functor $U : \mathbf{Op} \to \mathbf{Sig}$ defined by pulling back along the inclusion $\Lambda(\mathbf{Dep})^\dagger \to \Lambda(\mathbf{Dep})$.

\begin{definition}
Define a category $\mathbf{DepSet}$ as in \cref{dependent_set_over_signature_affine}, using $\Lambda(\mathbf{Dep})$ instead of $\Lambda(\mathbf{Dep_{aff}})$
\end{definition}

We still get a category:

\begin{proposition}
$\mathbf{DepSet}$ is a category: it admits identity morphisms and composition is associative.
\end{proposition}

\begin{proposition}
 Together with the projection functor $\mathbf{DepSet} \to \Lambda(\mathbf{Dep})$, the previous construction defines a dependent operad.
\end{proposition}

\begin{proof}
The proof of \cref{depset_cat} adapts directly, making the following observation: we have that, for $X$ an object of $\mathbf{DepSet}$ lying over $I$, that is $X \in \mathbf{Set}^I$, the object $Y := k_t^* X$ which lies over $J$, and together with the identity map $id_{k_t^* X} \in Hom(k_t^* X,k_t^*X)$ (where $k_t^* \simeq \eta_J^* k_*$), this provides a lifting of $k$ with domain $X$ that is clearly cocartesian.
\end{proof}

As in \cref{depset_sig}, we consider the following definition:

\begin{definition}
We write $\mathbf{DepSet}^*$ for the pullback in the diagram below:
\[\begin{tikzcd}[ampersand replacement=\&]
	{\mathbf{DepSet}^*} \&\& {\mathbf{DepSet}} \\
	\\
	{\Lambda(\mathbf{Dep})^\dagger} \&\& {\Lambda(\mathbf{Dep})}
	\arrow[from=1-1, to=1-3]
	\arrow[from=1-1, to=3-1]
	\arrow["\ulcorner"{anchor=center, pos=0.125}, draw=none, from=1-1, to=3-3]
	\arrow[from=1-3, to=3-3]
	\arrow[from=3-1, to=3-3]
\end{tikzcd}\]
\end{definition}

\begin{example}
 In this section, our running example is that of categories, we define $$(S_\mathbf{Cat}^\otimes)_0 \to \Lambda(\mathbf{Dep})^\dagger_0$$ as the fibration that has exactly two non-empty fibers, which are isomorphic to the one point set, over the one object poset and over the span $\bullet \leftarrow \bullet \rightarrow \bullet$. We denote the corresponding objects of $S_\mathbf{Cat}$ respectively by $o$ and $m$. 
 
 The dependent operad for categories $$T_\mathbf{Cat} : O_\mathbf{Cat} \to \Lambda(\mathbf{Dep})$$ has the same objects as $S_\mathbf{Cat}$ and is generated by two additional morphisms:
 \begin{itemize}
  \item $id : o \to m$ lying over the composite of the inert morphism that duplicate the point followed by the inclusion corresponding to the two codomains of the span legs.
  \item $comp : c \to m$ where $c$ is the (unique) object lying over the poset $P$ pictured below,
\[\begin{tikzcd}[ampersand replacement=\&]
	\& \bullet \&\& \bullet \& \\
	\bullet \&\& \bullet \&\& \bullet
	\arrow[from=1-2, to=2-1]
	\arrow[from=1-2, to=2-3]
	\arrow[from=1-4, to=2-3]
	\arrow[from=1-4, to=2-5]
\end{tikzcd}\]
lying over the functor obtained by composing the active inclusion of $P$ in
\[\begin{tikzcd}[ampersand replacement=\&]
	\& \textcolor{rgb,255:red,214;green,92;blue,92}{\bullet} \& \bullet \& \textcolor{rgb,255:red,214;green,92;blue,92}{\bullet} \& \\
	\bullet \&\& \textcolor{rgb,255:red,214;green,92;blue,92}{\bullet} \&\& \bullet
	\arrow[from=1-2, to=2-3]
	\arrow[from=1-3, to=2-1]
	\arrow[from=1-3, to=2-5]
	\arrow[from=1-4, to=2-3]
	\arrow[from=1-4, to=2-5]
	\arrow[from=2-1, to=1-2]
\end{tikzcd}\]
with the inert morphism that discards the three objects in red.
 \end{itemize}
These morphisms are subject to the usual equations for categories. 
\end{example}

We define a notion of algebra over an operad, and of dependent sets over a signature just like in \cref{algebra_affine} and \cref{dependent_set_over_signature_affine} respectively.

The proof of \cref{presheaf_topos} adapts directly to this new setting:

\begin{proposition}
 $\mathbf{DepSet}^{S}$ is equivalent to a presheaf topos.
\end{proposition}

\begin{example}
 The category of dependent sets for the signature corresponding to the operads of categories is easily computed to be the presheaf topos $\mathbf{Grph}$ of (non-reflexive) graphs.
\end{example}

The following theorem combines the results analogous to \cref{free_algebras}, \cref{monadicity} and \cref{algebras_loc_presentable}. The proofs are essentially the same.

\begin{theorem}
\label{main}
Given a dependent operad $O$ with underlying signature $S$, the forgetful functor $V^O$ admits a left adjoint $F^O :\mathbf{DepSet}^S \to \mathbf{Alg}_{O}$, yielding a monadic adjunction. 
 
 Moreover, $\mathbf{Alg}_{O}$ is locally presentable.
\end{theorem}

\begin{example}
 In the case of the operads for categories, we recover the usual free category functor from the category of graphs, and the associated monadicity of categories over graphs.
\end{example}

\begin{remark}
 Unlike affine dependent operads, whose (essentially) algebraic theories should not allow variable duplication in some sense, we expect the categories of algebras on a general dependent operad of this section to correspond exactly to the locally finitely presentable categories: we believe there is a canonical way to translate a general algebraic theory into a dependent operad.
\end{remark}

\begingroup
\setlength{\emergencystretch}{.5em}
\RaggedRight
\printbibliography
\endgroup

\end{document}